\documentclass[11pt,hyp,]{nyjm}
\usepackage{tikz}
\usetikzlibrary{matrix,arrows,decorations.pathmorphing}

\usepackage{hyperref} 
\hypersetup{nesting=true,debug=true,naturalnames=true}
\usepackage{xcolor}
\usepackage{amsmath,amsthm}
\usepackage{tikz-cd}
\usepackage{graphicx}
\usepackage[mathscr]{eucal}
\usepackage{amssymb}
\usepackage[frame,ps,matrix,arrow,curve,rotate,all,2cell,tips]{xy}
\usepackage{enumerate}
\usepackage{bbold}

\newtheorem{theorem}[subsection]{Theorem}
\newtheorem*{unnumthm}{Theorem}
\newtheorem{lemma}[subsection]{Lemma}

\newtheorem{proposition}[subsection]{Proposition}
\newtheorem{corollary}[subsection]{Corollary}

\newtheorem{prop}[subsection]{Proposition}

\theoremstyle{definition}
\newtheorem{definition}[subsection]{Definition}
\newtheorem{example}[subsection]{Example}
\newtheorem{remark}[subsection]{Remark}

\newtheorem{convention}[subsection]{Convention}
\newtheorem{notation}[subsection]{Notation}

\numberwithin{equation}{subsection}

\usepackage{tikz}
\usetikzlibrary{arrows,decorations.pathmorphing}
\usetikzlibrary{backgrounds,positioning}
\usetikzlibrary{fit,petri,shapes.misc}

\tikzset{auto}

\tikzset{empty/.style={circle,inner sep=0pt,minimum size=6mm}}
\tikzset{emptyvt/.style={circle,inner sep=0pt,minimum size=0mm}}

\tikzset{plain/.style={circle,draw,very thick,
inner sep=0pt,minimum size=6mm}}

\tikzset{fatplain/.style={rounded rectangle,draw,very thick,minimum size=6mm}}

\tikzset{bigplain/.style={rounded rectangle,draw,very thick,minimum size=.8cm}}

\tikzset{yellowvt/.style={circle,draw,fill=yellow,very thick,inner sep=0pt,minimum size=6mm}}

\tikzset{bluevt/.style={circle,draw,fill=blue!20,very thick,inner sep=0pt,minimum size=6mm}}

\tikzset{greenvt/.style={circle,draw,fill=green!30,very thick,inner sep=0pt,minimum size=6mm}}

\tikzset{redvt/.style={circle,draw,fill=red!30,very thick,inner sep=0pt,minimum size=6mm}}

\tikzset{arrow/.style={->,thick}}
\tikzset{dashedarrow/.style={->,dashed,thick}}
\tikzset{dottedarrow/.style={->,dotted,thick}}
\tikzset{mapto/.style={|->,thick}}

\tikzset{implies/.style={thick,double,double equal sign distance,-implies}}

\tikzset{line/.style={thick}}
\tikzset{dottedline/.style={dotted,thick}}
\tikzset{dashedline/.style={dashed,thick}}

\tikzset{inputleg/.style={<-,thick}}
\tikzset{outputleg/.style={->,thick}}
\tikzset{dottedinput/.style={<-,dotted,thick}}

\newcommand{\adjoint}{\rightleftarrows}
\newcommand{\nicearrow}{\SelectTips{cm}{10}}
\newcommand{\nicexy}{\nicearrow\xymatrix@C+5pt}

\newcommand{\boxprod}{\mathbin\square}
\newcommand{\comp}{\circ}
\newcommand{\defn}{\overset{\mathrm{def}}{=\joinrel=}}

\newcommand{\dotover}[1]{\underset{#1}{\centerdot}}

\newcommand{\Gr}{\mathtt{Gr}}
\newcommand{\gup}{\Gr^{{\scalebox{.5}{$\uparrow$}}}}

\newcommand{\gupc}{\gup_{\mathrm{c}}}
\newcommand{\gupd}{\gup_{\mathrm{di}}}

\newcommand{\wheel}{{\scalebox{.6}{$\circlearrowright$}}}

\newcommand{\gwheelc}{\Gr^\wheel_{\mathrm{c}}}

\newcommand{\tg}{\mathtt{G}}
\newcommand{\graphgpd}{\cg = (S,\tg)}

\newcommand{\ULin}{\mathtt{ULin}}

\newcommand{\UTree}{\mathtt{UTree}}
\newcommand{\Treew}{\mathtt{Tree}^\wheel}
\newcommand{\Treewheel}{\Treew}

\newcommand{\frakC}{\mathfrak{C}}
\newcommand{\fC}{\mathfrak{C}}

\newcommand{\sds}{\mathsf{ds}}

\newcommand{\Edge}{\mathsf{Edge}}

\newcommand{\sI}{\mathsf{I}}
\newcommand{\sJ}{\mathsf{J}}

\newcommand{\tensorunit}{\mathbb{1}}

\newcommand{\sn}{\mathsf{n}}
\newcommand{\ang}{A\bigl(\sn(G)\bigr)}
\newcommand{\lang}{(LA)\bigl(\sn(G)\bigr)}
\newcommand{\bng}{B\bigl(\sn(G)\bigr)}

\newcommand{\sO}{\mathsf{O}}

\newcommand{\opg}{\mathsf{Op}^{\cG}}

\newcommand{\ugbar}{\overline{U}_{\cG}}

\newcommand{\Lg}{L^{\cG}}

\newcommand{\Vt}{\mathsf{Vt}}

\newcommand{\cG}{\mathcal{G}}
\newcommand{\cg}{\mathcal{G}}
\newcommand{\sigmag}{\Sigma_{\cG}}

\newcommand{\ua}{\underline{a}}

\newcommand{\uc}{\underline{c}}

\newcommand{\ud}{\underline{d}}

\newcommand{\ur}{\underline{r}}

\newcommand{\smallop}{{\scalebox{.5}{$\mathrm{op}$}}}

\newcommand{\calm}{\mathcal{M}}
\newcommand{\M}{\calm}
\newcommand{\calmc}{\calm^{\fC}}

\newcommand{\Mtor}{\M^{\sigmabrr}}
\newcommand{\Mtos}{\M^{S}}

\newcommand{\caln}{\mathcal{N}}
\newcommand{\N}{\caln}

\newcommand{\Ntor}{\N^{\sigmabrr}}
\newcommand{\Ntos}{\N^{S}}

\newcommand{\sset}{\mathsf{sSet}}

\newcommand{\alg}{\mathsf{Alg}}

\renewcommand{\prop}{\mathsf{Prop}}

\newcommand{\gprop}{\prop^{\cG}}
\newcommand{\gpropm}{\gprop_{\M}}
\newcommand{\gpropn}{\gprop_{\N}}

\newcommand{\pofc}{\Sigma_{\frakC}}
\newcommand{\pofcop}
{\pofc^{\scalebox{.6}{$\mathrm{op}$}}}

\newcommand{\smallprof}[1]
{\raisebox{.05cm}{\scalebox{0.8}{#1}}}

\newcommand{\smallbinom}[2]
{\raisebox{.05cm}{\scalebox{0.8}{$\binom{#1}{#2}$}}}

\newcommand{\profv}{\smallprof{$\binom{\out(v)}{\inp(v)}$}}

\newcommand{\ccsingle}
{\smallprof{$\binom{c}{c}$}}

\newcommand{\dc}
{\smallprof{$\binom{\ud}{\uc}$}}
\newcommand{\brdc}
{\smallprof{$\binom{[\ud]}{[\uc]}$}}
\newcommand{\brdch}{([\uc];[\ud])}

\newcommand{\emptyprof}
{\smallprof{$\binom{\varnothing}{\varnothing}$}}

\newcommand{\inp}{\mathsf{in}}
\newcommand{\out}{\mathsf{out}}

\newcommand{\profilev}
{\smallprof{$\binom{\out(v)}{\inp(v)}$}}

\newcommand{\smallbr}[1]
{\raisebox{.03cm}{\scalebox{0.5}{#1}}}

\newcommand{\sigmabra}{\Sigma_{\smallbr{$[\ua]$}}}

\newcommand{\sigmabrc}{\Sigma_{\smallbr{$[\uc]$}}}

\newcommand{\sigmabrd}{\Sigma_{\smallbr{$[\ud]$}}}
\newcommand{\sigmabrr}{\Sigma_{\smallbr{$[\ur]$}}}

\newcommand{\sigmabrcop}{\sigmabrc^{\smallop}}

\newcommand{\sigmabrcopsigmabrd}{\sigmabrcop \times \sigmabrd}

\newcommand{\dch}{(\uc;\ud)}

\newcommand{\vertex}{\mathsf{Vt}}

\renewcommand{\lim}{\mathsf{lim}\,}

\DeclareMathOperator*{\colim}{\mathsf{colim}\,}

\DeclareMathOperator{\Hom}{Hom}
\DeclareMathOperator{\Aut}{Aut}

\DeclareMathOperator{\id}{id}
\DeclareMathOperator{\Id}{Id}

\newcommand{\alphabar}{\overline{\alpha}}

\begin{document}

\title[Shrinkability, relative left properness, derived base change]{Shrinkability, relative left properness, and derived base change}

\author{Philip Hackney, Marcy Robertson, and Donald Yau}
\address{Matematiska institutionen\\ Stockholms universitet\\ 106 91 Stockholm \\ Sweden}
\curraddr{Department of Mathematics\\ Macquarie University\\ NSW 2109 \\ Australia}
\email{hackney@math.su.se} 
\address{School of Mathematics and Statistics \\ The University of Melbourne \\ Victoria 3010 \\ Australia}\email{marcy.robertson@unimelb.edu.au}
\address{Department of Mathematics \\ The Ohio State University at Newark \\ Newark, OH}
\email{dyau@math.osu.edu}

\keywords{wheeled properads, operads, dioperads, model categories, left proper}

\subjclass[2010]{55U35, 18D50, 18G55, 55P48, 18D20}

\begin{abstract}
For a connected pasting scheme $\cG$, under reasonable assumptions on the underlying category, the category of $\mathfrak{C}$-colored $\cG$-props admits a cofibrantly generated model category structure.  In this paper, we show that, if $\cG$ is closed under shrinking internal edges, then this model structure on $\cG$-props satisfies a (weaker version) of left properness.  
Connected pasting schemes satisfying this property include those for all connected wheeled graphs (for wheeled properads), wheeled trees (for wheeled operads), simply connected graphs (for dioperads), unital trees (for symmetric operads), and unitial linear graphs (for small categories).  The pasting scheme for connected wheel-free graphs (for properads) does \emph{not} satisfy this condition. 

We furthermore prove, assuming $\cG$ is shrinkable and our base categories are nice enough, that a weak symmetric monoidal Quillen equivalence between two base categories induces a Quillen equivalence between their categories of $\cG$-props. The final section gives illuminating examples that justify the conditions on base model categories. 

\end{abstract}

\maketitle

\tableofcontents

\section{Introduction}

There has been significant recent work in determining when certain categories of operad-like objects admit Quillen model category structures \cite{jy1,muro11,fresse,dk,bb14}, or more generally determining when algebras over operad-like objects \cite{bm07,jy1} admit model category structures. 
All of these results require that the ground category $\M$ is well behaved; that is, when $\M$ is a cofibrantly generated monoidal model category, every object of $\M$ is cofibrant, and $\M$ admits functorial path data. 
Examples of such $\M$ include simplicial sets, symmetric spectra, and chain complexes in characteristic zero. In particular, the category of wheeled properads in a well-behaved model category carries a model category structure. 

In this paper we are not much concerned with the \emph{existence} of such model structures (we actually abstract it away in definition \ref{gprop-model-operad}), but rather in the \emph{properties} of the model structures when they exist.

The main property we investigate is that of (relative) left properness--that is to say we wish to know if equivalences between generalized props are closed under cobase change along cofibrations. Knowing if the model category structure on categories of generalized props satisfies a (relative) left properness result has many immediate applications. As an example, Bousfield localization is the process by which one adds weak equivalences to a model category while keeping the cofibrations fixed; in recent years it has become recognized as a fundamental tool in homotopy theory and in the related theory of $\infty$-categories. While it is not true that all model categories can be localized, an essential ingredient is that the initial model category be left proper. 

In the case of operads, there has been quite a bit of recent work on the question of relative left properness. If we are discussing \emph{non-symmetric} operads, then this model structure is (relatively) left proper \cite[Theorem 1.11]{muro14}, meaning that weak equivalences are preserved by pushouts along cofibrations between $\Sigma$-cofibrant objects.
A result of Rezk~\cite{rezkSA}, shows that that the category of symmetric operads is Quillen-equivalent to a left proper model category. 
In \cite{hryoperads}, however, we provided an example of Dwyer which shows that left-properness does \emph{not} hold for the category of symmetric operads itself. However, the category of symmetric operads does satisfy a weaker notion called \emph{relative left properness} \cite{hryoperads} and \cite[12.1.11]{fresse_book}.
The goal of this paper is to generalize this to other types of operad-like objects.

The main such operad-like structure that we address in this paper is that of wheeled properads \cite{mms} (although our theorems also apply to dioperads \cite{gan} and wheeled operads, in addition to the previously known cases of operads and categories).
Wheeled properads control (bi)-algebraic structures with traces, such as Frobenius algebras and unimodular Lie 1-bialgebras.
These ideas have proven fruitful in geometric situations related to mathematical physics, for instance to formal quasi-classical split quantum BV manifolds \cite{merkulov3}.

The following is a special case of our first main theorem \ref{gprop-relative-left-proper}.

\begin{unnumthm}
	Suppose that $\M$ is a ``nice enough'' monoidal model category (Definition~\ref{def_especially_nice}). 
	Then the model category structure on the category of wheeled properads in $\M$ is relatively left proper, in the sense that whenever we have a pushout diagram of wheeled properads
	\[
	\xymatrix{ A\,\, \ar@{>->}[r]  \ar@{->}[d]_-\simeq & X \ar@{->}[d] \\
	B\,\, \ar@{>->}[r] & Y
	}\] 
	with $A$ and $B$ both $\Sigma$-cofibrant and $A\to X$ a cofibration, the map $X\to Y$ is also a weak equivalence.  
\end{unnumthm}

A similar theorem holds for dioperads, wheeled operads, and so on. In order to justify the restrictions we place on our monoidal model categories we include in Section~\ref{appendix} several nontrivial examples to illuminate when (relative) left properness fails.

As soon as we are examining wheeled properads over general base categories, we can ask what happens when we modify the underlying category. Schwede and Shipley showed that, in favorable situations, a Quillen equivalence of the base categories induces a Quillen equivalence on the categories of monoids \cite{ss03}. 
Muro extended this in \cite[Theorem 1.1]{muro14} to show that such a Quillen equivalence also induces a Quillen equivalence on the category of non-symmetric operads.
The following is a special case of our second main theorem \ref{derived-basechange}.

\begin{unnumthm}
	Suppose that $\M$ and $\N$ are ``nice enough'' monoidal model categories (Definition~\ref{def_especially_nice}).
	Then a weak symmetric monoidal Quillen equivalence $\M \adjoint \N$ induces a Quillen equivalence between the associated categories of wheeled properads.
\end{unnumthm}
The right adjoint of the induced Quillen equivalence is defined levelwise, while the left adjoint is more subtle; this will be carefully constructed in section \ref{section basechange}.

Throughout this paper, we use the language of generalized props from \cite{jy2} which allows us to treat many cases simultaneously. This material is briefly covered in section~\ref{section_graphs_ps}. 
Two cases of interest which are \emph{not} addressed by the present paper are properads and props, as these cases do not satisfy a technical condition on pasting schemes, called ``shrinkability.'' Having our pasting schemes satisfy the shrinkability condition simplifies many constructions and arguments. It would be interesting to see if the corresponding results held for properads and props.

\subsection*{Related work:}
In a recent paper \cite{bb14}, Batanin and Berger develop similar results about the existence of a model category structure and its relative left properness using the framework of tame polynomial monads. A polynomial monad is a monad that encodes behavior similar to what we call pasting schemes. A polynomial monad is called tame if the ambient compactly generated model category is (strongly) $h$-monoidal. The main results of their paper imply that if a polynomial monad is tame, then the operad-like categories it encodes will have a (relative) left properness property. 

The techniques of this paper do not overlap significantly, and, in fact, should be considered in parallel.   In particular, the paper of Batanin and Berger show in \cite[Proposition 10.8, 10.9]{bb14} that the monad for wheeled properads is not tame. The model categories we consider in this paper are all $h$-monoidal \cite[Lemma 1.12]{bb14}, and we achieve a (relative) left properness result for wheeled properads. However, as we show in the final section, there also exist $h$-monoidal model categories, namely the category of simplicial sets with the usual Kan model structure, in which we don't have a (relative) left properness result. The conclusion is that, while tameness of a polynomial monad (or associated pasting scheme) is a good indicator of whether or not a category of operad-like objects is (relatively) left proper, it is not capturing the picture in its entirety.

\section{Preliminaries} 
\subsection{(Model) Categorical assumptions}

In this section we fix notation and definitions which our underlying categories satisfy. 

\begin{definition}[\cite{hirschhorn} (11.1.2)]
A model category $\M$ is said to be cofibrantly generated if there exist 
\begin{enumerate}
	\item a set $\sI$ of generating cofibrations which permits the small object argument \cite{hirschhorn} (10.5.15) such that a map is an acyclic fibration if and only if it has the right lifting property with respect to every element of $\sI$.
	\item a set $\sJ$ of generating acyclic cofibrations which permits the small object argument such that a map is a fibration if and only if it has the right lifting property with respect to every element of $\sJ$.
\end{enumerate}
\end{definition}

\begin{notation}
\label{convention-smc}
In general, when discussing closed symmetric monoidal categories \cite[VII]{maclane}, we will write $\otimes$ for the monoidal product, $\tensorunit$ for the tensor unit, and $\Hom(X,Y)$ for the internal hom object.
\end{notation}

The following definition is \cite{ss03} (3.1).

\begin{definition} Suppose that $\M$ is a model category. 
If $\M$ is also a closed symmetric monoidal category, we say that $\M$ is a \textbf{monoidal model category} if it satisfies
the following two axioms.
		\begin{description}
			\item[Pushout product] For each pair of cofibrations $f: A\to B$, $g: K\to L$, the map 
			\[
				f \boxprod g : A \otimes L \coprod_{A \otimes K} B \otimes K \to B \otimes L
			\]
			is also a cofibration. If, in addition, one of $f$ or $g$ is a weak equivalence, then so is $f\boxprod g$.
			\item[Unit] If $q: \tensorunit^c \overset\simeq\to \tensorunit$ is a cofibrant replacement of the unit object, then for every cofibrant object $A$,
			\[
				q \otimes \Id : \tensorunit^c \otimes A \to \tensorunit \otimes A \cong A
			\]
			is a weak equivalence.
		\end{description}
\end{definition}

\begin{convention}
For the rest of the paper, all model categories $\M$ will be monoidal model categories which are cofibrantly generated by some chosen sets of generating (acyclic) cofibrations $\sI$ (resp. $\sJ$).
\end{convention}

\subsection{Graphs and pasting schemes}\label{section_graphs_ps}

We will briefly give the necessary definitions and notations regarding colored objects in $\calm.$ A more complete discussion of the following definitions can be found in \cite{jy2}.

\begin{definition}[Colored Objects]
\label{def:profiles}
Fix a non-empty set  of \textbf{colors}, $\fC$. 
\begin{enumerate}
\item
A \textbf{$\fC$-profile} is a finite sequence of elements in $\fC$,
\[
\uc = (c_1, \ldots, c_m) = c_{[1,m]}
\]
with each $c_i \in \fC$.  If $\fC$ is clear from the context, then we simply say \textbf{profile}. The empty $\fC$-profile is denoted $\varnothing$, which is not to be confused with the initial object in $\calm$.  Write $|\uc|=m$ for the \textbf{length} of a profile $\uc$.
\item
An object in the product category $\prod_{\fC} \calm = \calm^{\fC}$ is called a \textbf{$\fC$-colored object in $\calm$}; similarly a map of $\fC$-colored objects is a map in $\prod_{\fC} \calm$.  A typical $\fC$-colored object $X$ is also written as $\{X_a\}$ with $X_a \in \calm$ for each color $a \in \fC$.
\item 
Fix $c \in \fC$.  An $X \in \calmc$ is said to be \textbf{concentrated in the color $c$} if $X_d = \varnothing$ for all $c \not= d \in \fC$.
\item
For $f : X \to Y \in \calm$ we say that $f$ is said to be \textbf{concentrated in the color $c$} if both $X$ and $Y$ are concentrated in the color $c$.
\end{enumerate}
\end{definition}

\begin{definition}
A \emph{graph} $G$ consists of
\begin{itemize}
\item a directed, connected, non-empty graph $G$ with half-edges (also called `flags'), 
\item listings on the inputs and outputs of the graph
\[
	\ell_G : (\inp G; \out G) \overset{\cong}\to (1, \dots, |\inp G|; 1, \dots, |\out G|)  \]
and
\item listings on the inputs and outputs of each vertex $v\in \vertex(G)$
\[
	\ell_v : (\inp v; \out v) \overset{\cong}\to (1, \dots, |\inp v|; 1, \dots, |\out v|)  \]
\end{itemize}
If, in addition, we have a coloring function $\xi: \Edge(G) \to \fC$ to some set $\fC$, then we say that $G$ is a \emph{$\fC$-colored graph}.
	\begin{itemize}
		\item A \textbf{weak isomorphism} $G\to G'$ between two $\fC$-colored graphs is a isomorphism which preserves the graph structure and the coloring, but not necessarily the listings.
		\item The collection of $\fC$-colored graphs with weak isomorphisms forms a groupoid which we denote by $\gwheelc (\fC)$.
	\end{itemize}
\end{definition}

\begin{figure}
	\includegraphics[height=0.65in]{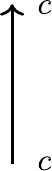} \qquad \includegraphics[width = .75in]{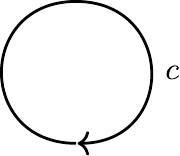} \qquad \includegraphics[width=.75in]{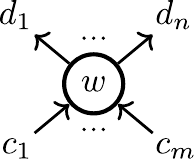}
	\caption{An exceptional edge, an exceptional loop, and a corolla
	(graphics from \cite{jy2})
	}\label{figure examples}
\end{figure}

\begin{example} See figure \ref{figure examples}.
	\begin{itemize}
		\item If $c\in \fC$, then there is a $\fC$-colored graph $\uparrow_c$ which has profiles $(c;c)$, no vertices, and a single edge.
		\item If $c\in \fC$, then there is a $\fC$-colored graph $\circlearrowright_c$ which has profiles $(\varnothing; \varnothing)$, no vertices, and a single $c$-colored edge.
		\item If $\dch = (c_{[1,m]}; d_{[1,n]})$ is a pair of $\fC$-profiles, then there is \textbf{standard corolla} $C_{\dch}$ with 
		\begin{itemize}
			\item one vertex $v$;
			\item $m+n$ flags, with inputs $\{ 1^i, 2^i, \dots, m^i\}$ and outputs $\{ 1^o, \dots, n^o\}$; 
			\item $\ell_G(k^i) = \ell_v(k^i) = k$ and $\ell_G(k^o) = \ell_v(k^o) = k$; and
			\item $\xi(k^i) = c_k$ and $\xi(k^o) = d_k$.
		\end{itemize}
	\end{itemize}
\end{example}

An \emph{ordinary} internal edge is an edge that is neither an exceptional edge $\uparrow$ nor an exceptional loop $\circlearrowright$.

\begin{definition}[Graph operations] Suppose that $G$ is a $\fC$-colored graph with profiles $\dch$.
	\begin{itemize}
	\item If, for each $v\in \vertex(G)$, $H_v$ is a graph with profiles $\profilev$, then the \textbf{graph substitution} 
	\[
		G\{H_v\}_{v\in\vertex(G)}
	\]
	is the graph obtained from $G$ by
	\begin{itemize}
		\item replacing each vertex $v\in \vertex(G)$ with the graph $H_v$, and
		\item identifying the legs of $H_v$ with the incoming/outgoing flags of $v$.
	\end{itemize}
	See figure \ref{figure graph substitution}.
	\item The input extension $G_{in}$ is the graph with profiles $\dch$ where we graft a corolla $C_{(c_i; c_i)}$ onto the input leg $\ell_G^{-1} (i)$. See figure \ref{figure extensions}.
	\item The output extension $G_{out}$ is the graph with profiles $\dch$ where we graft a corolla $C_{(d_i; d_i)}$ onto the output leg $\ell_G^{-1} (i)$.
	\end{itemize}
\end{definition}

\begin{figure}
	\includegraphics[scale=0.4]{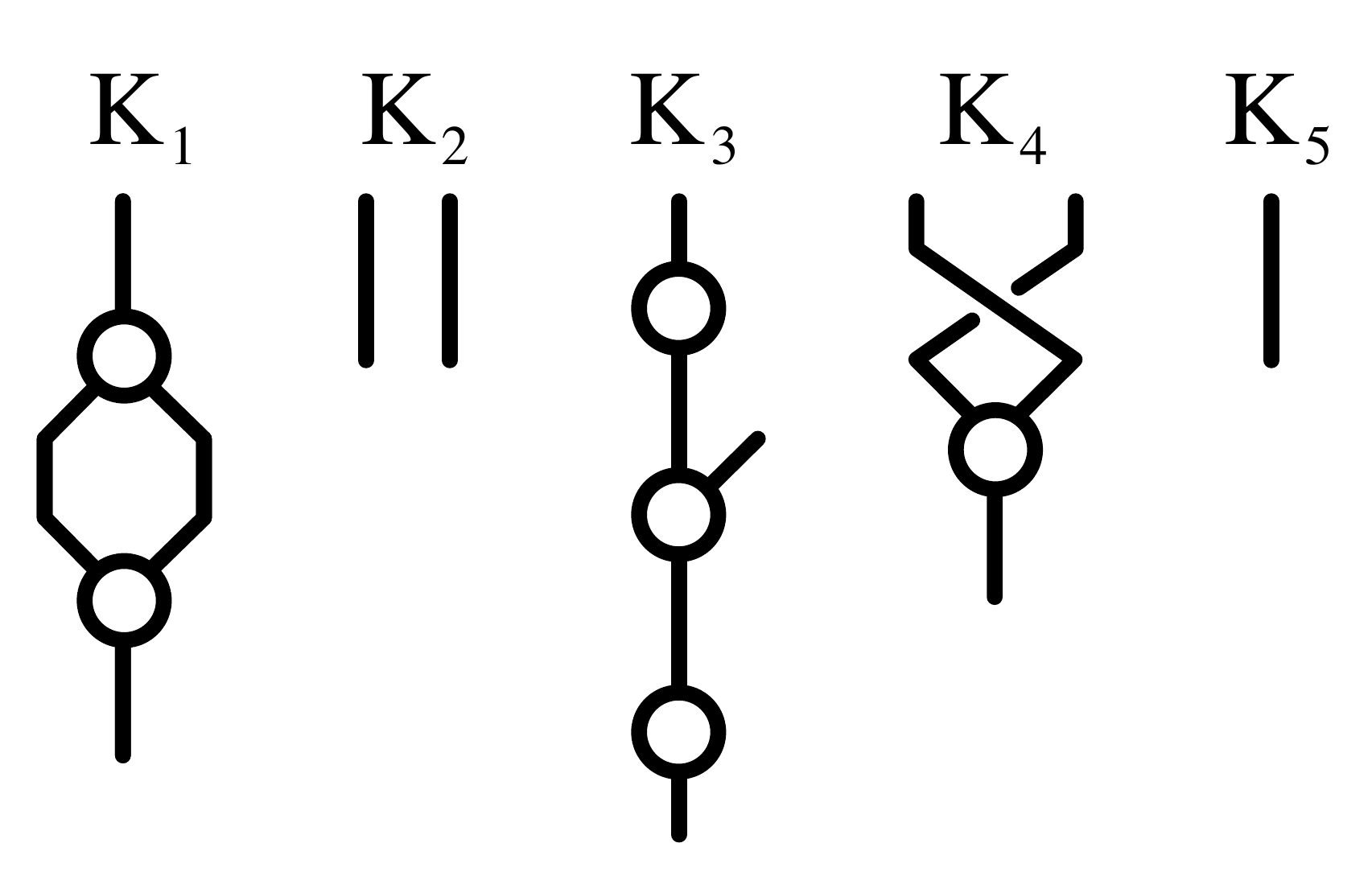}
	\includegraphics[scale=0.4]{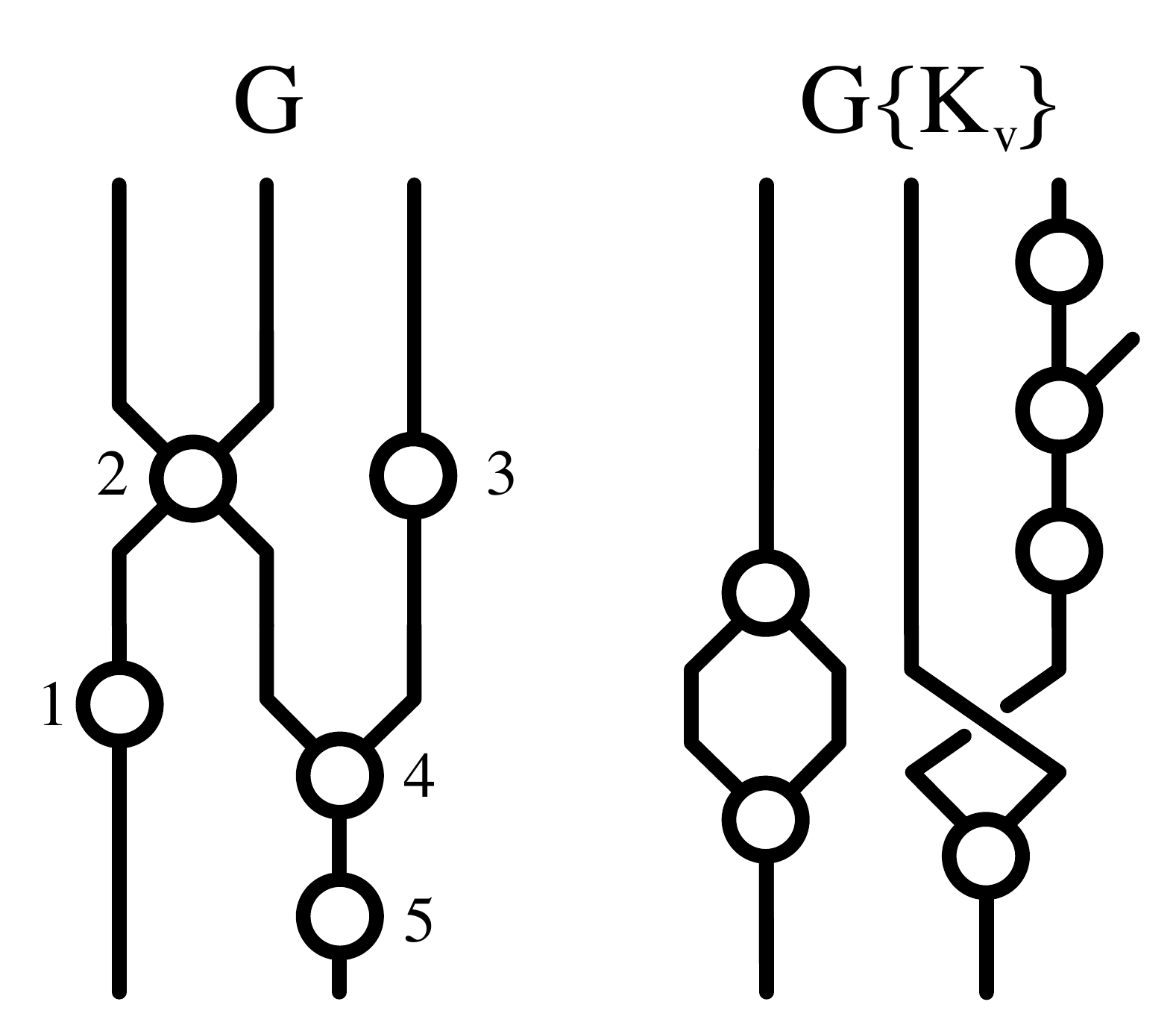}
	\caption{Graph substitution (from \cite{hr2})}\label{figure graph substitution}
\end{figure}

\begin{definition}[Groupoid of profiles]
Let $\fC$ be a non-empty set.  Given a $\fC$-profile $\ua$ and an element $\sigma\in\Sigma_{|\ua|}$ then define 
\[\sigma\ua = (a_{\sigma^{-1}(1)}, \ldots , a_{\sigma^{-1}(m)})\] and  

\[
\ua\sigma = (a_{\sigma(1)}, \ldots , a_{\sigma(m)}). 
\]

The groupoid of $\fC$-profiles, denoted $\pofc$, has objects $\fC$-profiles and morphisms left permutations $\ua\to \sigma\ua$. The opposite groupoid of $\fC$-profiles, denoted $\pofcop$, has objects $\fC$-profiles and morphisms right permutations $\ua\to \ua\sigma$. 
\end{definition}

A subcategory $C'$ of a category $C$ is called \textbf{replete} if, for each object $c' \in C'$ and each isomorphism $c' \to c$ in $C$, the object $c$ is also in $C'$.

\begin{definition}\label{def_pasting_scheme}[Pasting Scheme]
A $\fC$-colored (connected, unital) pasting scheme is a pair
\[ \graphgpd \]
in which
\begin{enumerate}
	\item $S$ is a replete and full subgroupoid of $\pofcop \times \pofc$, and
	\item $\tg$ is a replete and full subgroupoid of $\gwheelc(\fC)$
\end{enumerate}
such that:
\begin{enumerate}
	\item if $\tg$ is in $\cg$ and $(\uc;\ud)$ is the input-output profile of $\tg$, then $(\uc;\ud)\in S$,
	\item $\cg$ is closed under graph substitution,
	\item $\tg$ contains all the $\dch$-corollas for all pairs of profiles $\dch \in S$,
	\item if $\dch \in S$ and $C=C_{\dch}$ is the $\dch$-corolla, then the input extension $C_{in}$ and the output extension $C_{out}$ are both in $\tg$ (see figure \ref{figure extensions}), and
	\item if $c\in \fC$, then $(c; c) \in S$ and $\uparrow_c~\in \tg$.
\end{enumerate}
\end{definition}

\begin{figure}
	\includegraphics[width=0.8\textwidth]{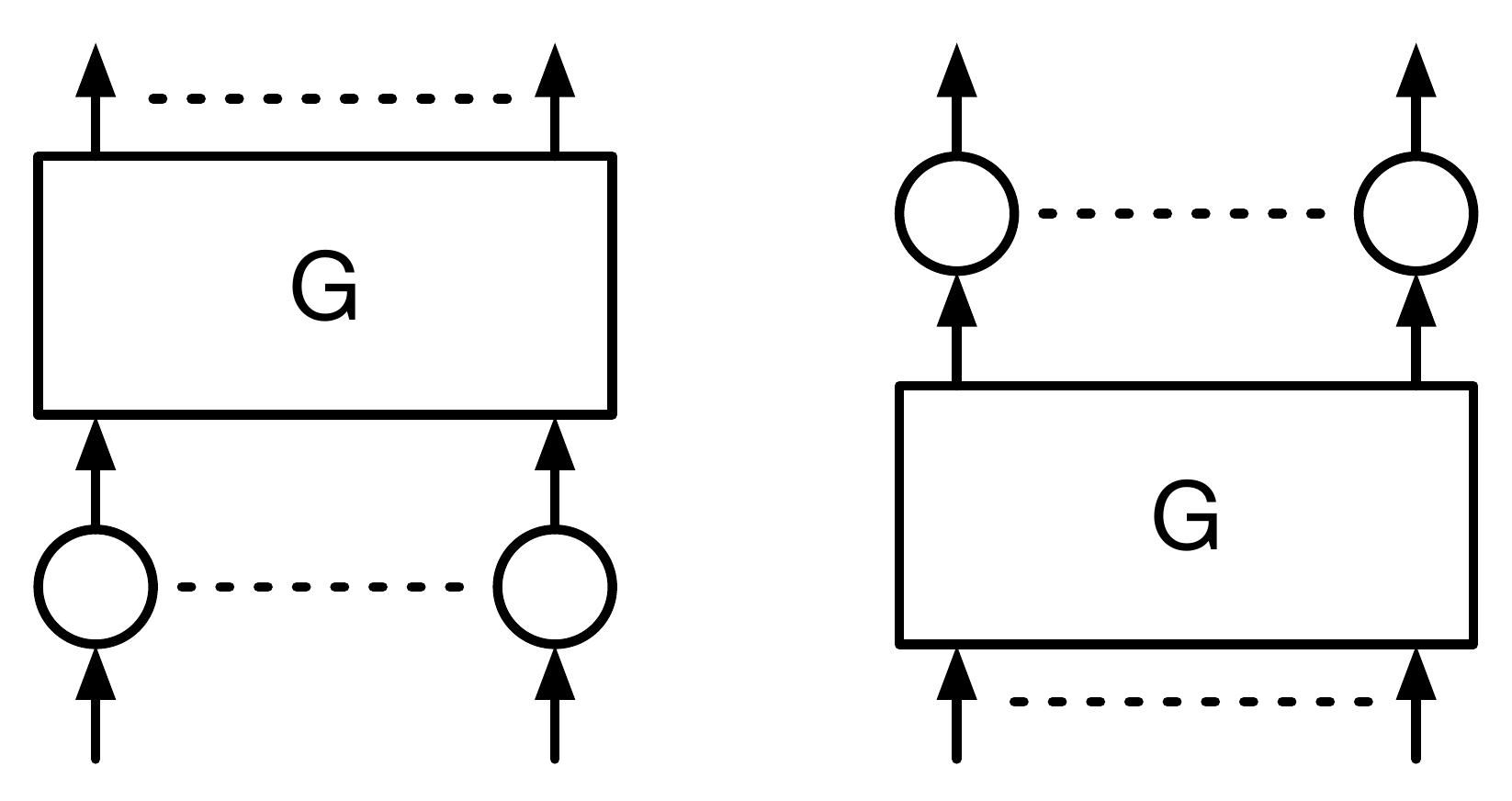} 
	\caption{
	The input extension $G_{in}$ and output extension $G_{out}$
	}
	\label{figure extensions}
\end{figure}

\begin{definition}[Colored Symmetric Sequences]
\label{def:colored-sigma-object}
Fix a non-empty set $\fC$ and $\graphgpd$ is a $\fC$-colored pasting scheme \cite{jy2} (Def. 8.2).
\begin{enumerate}
\item
The \textbf{orbit} of a profile $\ua$ is denoted by $[\ua]$.  The maximal connected sub-groupoid of $\pofc$ containing $\ua$ is written as $\sigmabra$.  Its objects are the left permutations of $\ua$.  There is an \textbf{orbit decomposition} of $\pofc$
\begin{equation}
\label{pofcdecomp}
\pofc \cong \coprod_{[\ua] \in \pofc} \sigmabra,
\end{equation}
where there is one coproduct summand for each orbit $[\ua]$ of a $\fC$-profile. 
\item
Consider the diagram category $\M^S$, whose objects are called \textbf{$\sigmag$-objects in $\M$}.  The decomposition \eqref{pofcdecomp} implies that there is a decomposition
\begin{equation}
\label{m-tothe-s-again}
\M^{S}
\cong 
\prod_{([\uc];[\ud]) \in S} \M^{\sigmabrcopsigmabrd}.
\end{equation}
\item
For $X \in \M^{S}$, we write
\begin{equation}
\label{sigmacopd-component}
X\brdc \in \calm^{\sigmabrcopsigmabrd}
\end{equation}
for its $\brdch$-component.  For $\dc \in S$ (both $\uc$ and $\ud$ are $\fC$-profiles), we write
\begin{equation}
\label{dc-component}
X\dc \in \calm
\end{equation}
for the value of $X$ at $\dch$.
\end{enumerate}
\end{definition}

Unless otherwise specified, we will assume that $\fC$ is a fixed, non-empty set of colors, and $\graphgpd$ is a $\fC$-colored pasting scheme.

\subsection{$\cG$-Props as colored operadic algebras} 
For a $\fC$-colored operad $\sO$ in $\M$ (see \cite{jy2} 11.14 or \cite{yau-operad}  11.2.1), denote by $\alg(\sO)$ the category of $\sO$-algebras (\cite{jy2} 13.37 or \cite{yau-operad} 13.3.2).  Limits of $\alg(\sO)$ are taken in the underlying category of $\fC$-colored objects $\calmc$ via the free-forgetful adjoint pair
\[
\nicexy{
\calmc \ar@<2pt>[r]^-{\sO \comp -} 
& \alg(\sO). \ar@<2pt>[l]
}\]
Here $\circ$ is the $\fC$-colored version of the circle product for operads, which in single-colored form first appeared in \cite{rezk}. 
Detailed description of the general colored version of $\circ$ can be found in \cite{white-yau} (3.2).

The category of $\fC$-colored objects, $\calmc$, admits a cofibrantly generated model category structure where weak equivalences, fibrations, and cofibrations are defined entrywise, as described in \cite{hirschhorn} (11.1.10).  
In this model category a generating cofibration in $\calmc = \prod_{\fC} \calm$ (i.e., a map in $\sI$) is a generating cofibration of $\calm$, concentrated in one entry.  Similarly, the set of generating acyclic cofibrations is $\sJ \times \fC$.  In addition, the properties of being simplicial, or proper, are inherited from $\calm$.

Recall, from Section 14.1 of \cite{jy2}, the $|S|$-colored operad $\opg$ (there called $\ugbar$) controlling $\cg$-props, where $|S|$ is the set of objects of the groupoid $S$. The elements of $\opg$ are graphs with ordered sets of vertices, and operadic composition is given by graph substitution. 
The following is Lemma 14.4 in \cite{jy2}.

\begin{lemma}
\label{gprops-are-algebras}
Suppose $\M$ is a bicomplete symmetric monoidal closed category, and  $\graphgpd$ is a pasting scheme \cite{jy2} (Def. 8.2).  Then there is an $|S|$-colored operad $\opg$ in $\M$ such that there is a canonical isomorphism of categories
\[
\gprop \cong \alg(\opg)
\]
between:
\begin{itemize}
\item
the category $\gprop$ of $\cG$-props in $\M$ \cite{jy2} (Def. 10.39) and 
\item
the category $\alg(\opg)$ of algebras over the colored operad $\opg$ in $\M$.
\end{itemize}
\end{lemma}

\begin{definition}
\label{gprop-model-operad}
Suppose that $\calm$ is a monoidal model category which is cofibrantly generated with set of generating cofibrations $\sI$ and set of generating acyclic cofibrations $\sJ$.
If $\cG$ is a pasting scheme, we say \textbf{$\cG$ is admissible in $\calm$} if $\gpropm = \gprop \cong \alg(\opg)$ admits a cofibrantly generated model structure\footnote{We say admissible here since, due to the conditions on our base category $\M$, all colored operads $\opg$ are admissible in the language of \cite{bm07}.}, in which:
\begin{itemize}
\item
fibrations and weak equivalences are created entrywise in $\calm$, and
\item
the set of generating (acyclic) cofibrations is $\opg \comp \sI$ (resp., $\opg \comp \sJ$), where $\sI$ (resp., $\sJ$) is the set of generating (acyclic) cofibrations in $\M^{|S|}$.
\end{itemize}
\end{definition}

Bicompleteness is automatic by \cite{white-yau} (4.2.1), with reflexive coequalizers and filtered colimits preserved and created by the forgetful functor $\gprop \to \M^{|S|}$.
For certain choices of $\M$, such as compactly generated Hausdorff spaces, simplicial sets, and symmetric spectra (with the projective model structure) we know that each pasting scheme $\cG$ is admissible in $\M$ by \cite{bm07} (2.1).

\begin{example}
If $k$ is a characteristic zero field and $\M = \operatorname{Ch}(k)$ or $\operatorname{Ch}_{\geq 0}(k)$ with the projective model structure (\cite[2.3.11]{hovey}, \cite[4.12]{quillen}, respectively), then $\M$ is admissible for every $\cG$.
In the unbounded case, Fresse shows in \cite[5.3]{fresse} that $\operatorname{Ch}(k)$ admits `functorial path data', which combined with Lemma \ref{lemma semisimple} shows that $\M$ is nice in the sense of \cite[2.6.6]{hryoperads}. 
Thus every $\cG$ is admissible in $\operatorname{Ch}(k)$ by \cite[2.6.8]{hryoperads}. 

For the non-negatively graded case, consider the truncations from \cite[1.2.7]{weibel} of a chain complex $C$:
\[
	(\tau_{\geq n} C)_i = \begin{cases}
		0 & \text{if }i<n \\
		Z_nC & \text{if }i=n \\
		C_i & \text{if }i > n.
	\end{cases}
\]
These have the property that the inclusion $\tau_{\geq n} C \to C$ is a chain map which is an isomorphism on $H_i$ for $i\geq n$.
Since $\operatorname{Ch}(k)$ admits functorial path data, then so does $\operatorname{Ch}_{\geq 0}(k)$ by defining the path object of $C \in \operatorname{Ch}_{\geq 0}(k)$ to be $\tau_{\geq 0} \operatorname{Path}(C)$.
The result again follows from \cite[2.6.8]{hryoperads}.
\end{example}

We believe the following lemma is well-known, but were unable to find a proof in the literature.

\begin{lemma}\label{lemma semisimple}
If $R$ is a semisimple ring, then every object of $\operatorname{Ch}(R)$ is cofibrant.
\end{lemma}
\begin{proof}
For a given $n$, the map $\tau_{\geq n} C \to \tau_{\geq n-1} C$ is an isomorphism outside of degrees $n$ and $n-1$.
Thus the cokernel is bounded below; every $R$-module is projective so the cokernel is cofibrant by \cite[2.3.6]{hovey}. Moreover, every injection of $R$ modules splits, so the map is a dimensionwise split injection; hence, by \cite[2.3.9]{hovey}, $\tau_{\geq n} C \to \tau_{\geq n-1} C$ is a cofibration.

By \cite[2.3.6]{hovey}, $\tau_{\geq 0} C$ is cofibrant. 
Thus we have a sequence of cofibrations
\[
	0 \to \tau_{\geq 0} C \to \tau_{\geq -1} C \to \tau_{\geq -2} C \to \cdots
\]
and so $C = \colim\limits_{n\leq 0} \tau_{\geq n} C$ is cofibrant as well.
\end{proof}

\begin{remark}
	We have phrased everything above as a lifting of the model structure from $\M^{|S|}$.
	This is purely for convenience, to match the existing literature.
	One could instead regard $\opg$ as a kind of operad which is colored by the \emph{groupoid} $S$ (rather than the \emph{set} $|S|$) and most of the theory should still go through, lifting from the model structure on $\M^S$ (11.6.1) \cite{hirschhorn}.
	It seems that the benefits of doing so are minimal, as all maps in $\gpropm$ are already in $\M^S$ and the fibrations and weak equivalences in $\M^S$ are detected levelwise, just as in $\M^{|S|}$.
	We will thus usually assume we are working in the subcategory $\M^S$ rather than in $\M^{|S|}$ when it makes no difference.
\end{remark}

\section{Shrinkable pasting schemes}

\begin{convention}
The book \cite{jy2} is our general reference for graphs.  From this point forward, by a \emph{graph} we mean a strict isomorphism class \cite{jy2} (Def. 4.1) of a wheeled graph in the sense of \cite{jy2} (Def. 1.29).  Graph substitution of strict isomorphism classes of graphs is well-defined, strictly associative, and unital by \cite{jy2} (Lemma 5.31 and Theorem 5.32).
\end{convention}

Fix a non-empty set $\fC$ of colors once and for all.  As in Definition~\ref{def_pasting_scheme} all pasting schemes in this paper are connected and unital.    

\begin{figure}
	\includegraphics[height=1in]{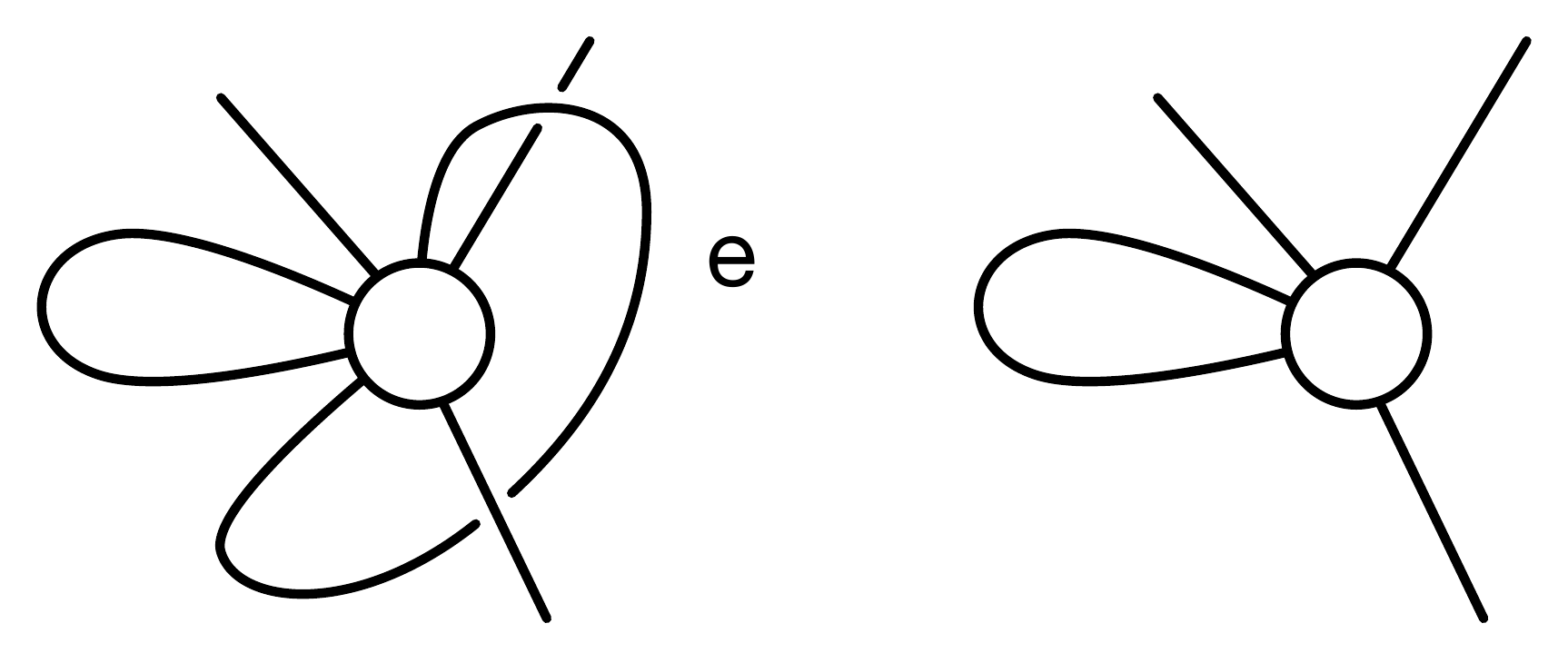}
	\caption{Shrinking away a loop $e$
	}\label{figure shrink loop}
\end{figure}

If $e$ is an internal edge in a graph $G$, then shrinking away $e$ results in a new graph $G'$. If $e$ is a loop, $G'$ is obtained by deleting $e$ from $G$ (Figure \ref{figure shrink loop}); otherwise we must first identify the two vertices of $e$ (Figure \ref{figure shrink edge}).
Note that shrinking an internal edge is not an operation which comes from graph substitution  (Example 6.9) \cite{jy2}.

\begin{figure}
	\includegraphics[height=1.8in]{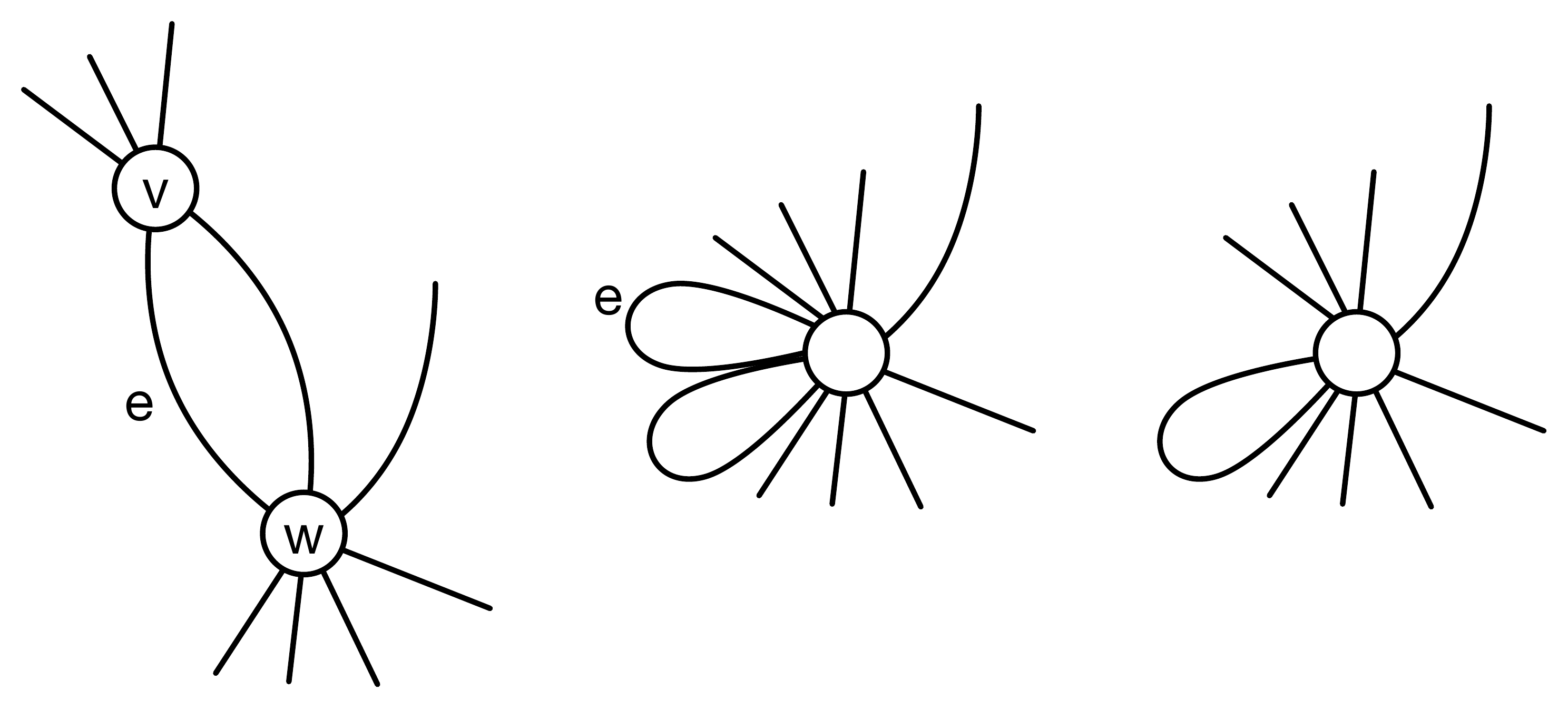}
	\caption{Shrinking away the internal edge $e$}\label{figure shrink edge}
\end{figure}

If $G \in \cG$ and $e$ is an ordinary internal edge of $G$, then shrinking $e$ may or may not give us a graph which is still in $\cG$; for example the chunk on the left in Figure \ref{figure shrink edge} may be part of a graph in $\gup$, while the chunk on the right cannot be.

\begin{definition}
\label{def:shrinkable}
Suppose given a $\fC$-colored pasting scheme $\graphgpd$, i.e., $\cG \leq \gwheelc (\fC)$ where $\gwheelc(\fC)$ is the collection all $\fC$-colored connected wheeled graphs.
Then $\cG$ is called \textbf{shrinkable} if it is closed under the operation of  shrinking an ordinary internal edge.
\end{definition}

The reader is referred to \cite{jy2} (Sec. 8.1) for notations regarding the following pasting schemes.

\begin{proposition}
\label{shrinkable-pasting-schemes}
The following $\fC$-colored pasting schemes are shrinkable.
\begin{enumerate}
\item
$\gwheelc (\fC) =$ connected wheeled pasting scheme (for wheeled properads)
\item
$\Treewheel (\fC) =$ wheeled tree pasting scheme (for wheeled operads)
\item
$\gupd (\fC) =$ simply-connected pasting scheme (for dioperads)
\item
$\UTree (\fC)=$ unital level trees pasting scheme (for operads)
\item
$\ULin (\fC) =$ unital linear pasting scheme (for small categories)
\end{enumerate}
\end{proposition}

\begin{proof}
For the first two pasting schemes, since loops and other directed cycles are allowed, they are closed under deleting loops and shrinking an internal edge with distinct end vertices (which usually results in new loops).
The last three pasting schemes are all contained in the simply-connected pasting scheme $\gupd (\fC)$.  That $\gupd(\fC)$ is shrinkable is \cite{jy2} (Lemma 6.8).  That the smaller pasting schemes are also shrinkable follows from the same argument with minor change of terminology.
\end{proof}

\begin{example}
\label{gupc-not-shrinkable}
The $\fC$-colored connected wheel-free pasting scheme $\gupc$ (for properads) is not shrinkable.  For example, in the walnut graph \cite{jy2} (Example 1.41)
\begin{center}
\begin{tikzpicture}
\matrix[row sep=.5cm,column sep=2cm] {
\node [plain] (v) {$v$};\\
\node [plain] (u) {$u$};\\
};
\draw [arrow, bend left=60] (u) to (v);
\draw [arrow, bend right=60] (u) to (v);
\end{tikzpicture}
\end{center}
if either internal edge is shrunk, then the result has a loop
\begin{center}
\begin{tikzpicture}
\matrix[row sep=1cm,column sep=.8cm] {
\node [plain] (v) {$uv$};\\
};
\draw [arrow, out=45, in=-45, looseness=7] (v) to (v);
\end{tikzpicture}
\end{center}
which does not belong to $\gupc$ anymore.
\end{example}

\begin{definition}
\label{def:strict-pasting-scheme}
Suppose given a $\fC$-colored pasting scheme $\graphgpd$.
\begin{enumerate}
\item
A \textbf{marked graph} in $\cG$ is a pair $(G,\sds)$ with 
\begin{itemize}
\item
$G \in \cG$ and 
\item
$\varnothing \not= \sds \subseteq \vertex(G)$.  
\end{itemize}
We will also say $(G,\sds) \in \cG$ is a marked graph.
\item
Suppose $(G,\sds)$ is a marked graph in $\cG$.
\begin{itemize}
\item
An element in $\sds$ is called a \textbf{distinguished vertex}.  
\item
Vertices in the complement
\[
\sn(G,\sds) = \sn(G) \defn \vertex(G) \setminus \sds
\]
are called \textbf{normal vertices}.
\end{itemize}
\item
A \textbf{weak isomorphism} $f :(G, \sds) \to (G',\sds')$ between marked graphs is defined as a weak isomorphism $f : G \to G' \in \cG$ that induces a bijection $f : \sds \cong \sds'$ between the sets of distinguished vertices.
A \textbf{well marked graph} is a marked graph $(G,\sds)$ in which every flag in a distinguished vertex is part of an internal edge whose other end vertex is normal.
\item
A \textbf{reduced marked graph} is a well marked graph in which there are no internal edges with both end vertices normal. 
\end{enumerate}
\end{definition}

\begin{remark}
\label{reduced-means}
For a non-exceptional loop in a graph, the end vertices are equal.  Thus, a reduced marked graph is precisely a marked graph that satisfies the following conditions:
\begin{enumerate}
\item
Every internal edge has one normal end vertex and one distinguished end vertex.  Notice that this condition implies that there are no loops at any vertex.
\item
Every input or output leg of the graph is adjacent to a normal vertex.
\end{enumerate}
For simply connected pasting schemes, wellness means that every distinguished vertex is bounded on both sides by normal vertices.
\end{remark}

The following observation ensures that well marked graphs are closed under graph substitution.

\begin{proposition}
\label{well-marked-graphsub}
Suppose:
\begin{itemize}
\item
$\cG$ is a pasting scheme, and $K \in \cG$ with $\Vt(K) \not= \varnothing$.
\item
For each $v \in \Vt(K)$, $(G_v, \sds_v) \in \cG$ is a marked graph such that $G_v$ has the same input/output profiles as $v$.  
\end{itemize}
Then:
\begin{enumerate}
\item
The graph substitution $H = K(\{G_v\}) \in \cG$ is canonically a marked graph with set of distinguished vertices
\[
\sds_H \defn \coprod_{v \in K} \sds_v.
\]
\item 
If each $(G_v,\sds_v)$ is a well marked graph, then $(H,\sds_H)$ is also a well marked graph.  
\end{enumerate}
\end{proposition}

\begin{proof}
The pair $(H,\sds_H)$ is a marked graph because $\Vt(K) \not= \varnothing$ and each $\sds_v \not= \varnothing$.  For the second assertion, to see that it is a well marked graph, first note that a distinguished vertex $w \in \sds_H$ must be a distinguished vertex in some unique $G_u$.  Since $G_u$ is a well marked graph, every flag in $w$ is part of an internal edge in $G_u$, hence an internal edge in $H$, whose other end vertex is normal in $G_u$, hence also normal in $H$.  In this last sentence, we used the equalities
\[
\sn(H,\sds_H) = \Vt(H) \setminus \sds_H = \coprod_{v \in K} \Vt(G_v) \setminus \coprod_{v \in K} \sds_v = \coprod_{v \in K} \sn(G_v,\sds_v)
\]
to identify the normal vertices in $H$ with those in the various $G_v$'s.
\end{proof}

\begin{remark}
\label{rk:may-not-reduce}
One must be careful that, in the context of the previous proposition, even if each $(G_v,\sds_v)$ is reduced, it does \emph{not} follow in general that $(H,\sds_H)$ is reduced.  In forming the graph substitution $H = K(\{G_v\})$, there are usually some internal edges that are not internal edges in any $G_v$.  These new internal edges come from  the legs of the $G_v$'s that are connected in $H$ by some internal edge in $K$.  So such a new internal edge in $H$ may connect a normal vertex $w_u$ in some $G_u$ with a normal vertex $w_v$ in some $G_v$, where $G_u = G_v$ and even $w_u = w_v$ are allowed if the corresponding internal edge in $K$ is a loop at $v$.  In particular, $H$ may not be reduced, although by the previous proposition it must be well marked.  It does have a unique reduction up to weak isomorphism, as we will see in Corollary \ref{graph-sub-reduced}.
\end{remark}

A pasting scheme $\cG$ is shrinkable if and only if, for each internal edge $e$ in an ordinary graph $G \in \cG$, the wheeled graph $G/e$ obtained from $G$ by shrinking away the internal edge $e$ is still in $\cG$.  In forming $G/e$, the two flags that make up $e$ are removed and the vertices to which they belong are redefined as a single vertex. When $e$ is a loop at a vertex, $G/e$ means $G$ with $e$ deleted.  The rest of the graph structures in $G/e$ is inherited from $G$.  In particular, when $e$ is an internal edge that is not a loop, the new combined vertex inherits the dioperadic listings from the two original end vertices of $e$ as in \cite{jy2} (2.4.2).

The operations of shrinking two internal edges in a given graph--depending on the order in which they are shrunk--is well-defined only up to weak isomorphism.  One can see this from, for example, \cite{jy2} (Lemma 6.106).  In trying to shrink two internal edges from two distinct vertices going into a third vertex, the outgoing listing of the combined vertex may need to be corrected with a block permutation, depending on the order in which the internal edges are shrunk.  This is why marked graphs are considered with \emph{weak isomorphisms} preserving the distinguished vertices.  

In general, if $E \subseteq \Edge(G)$ is a non-empty subset of ordinary internal edges, then $G/E$--that is, $G$ with all $e \in E$ shrunk--is uniquely defined up to vertex listings, but its graph listing must be that of $G$.  In other words, $G/E$ is the result of removing all the flags corresponding to all $e \in E$, combining all the affected vertices connected to each other into a single vertex, and taking as much graph structure from $G$ as possible.  The only structure that is not uniquely defined in such a $G/E$ is the set of vertex listings for the newly formed vertices.   Given any such choice of $G/E$, there is a unique graph substitution decomposition
\begin{equation}
\label{g-over-e}
G = (G/E)\bigl(\{H\}\bigr)
\end{equation}
in which the internal edges in the $H$'s form precisely the set $E$.  Choosing a different representative of $G/E$ (by changing some listings of the newly formed vertices) can only change the graph listings, but \emph{not} vertex listings, of the $H$'s.  So given $E \subseteq \Edge(G)$, the $H$'s are uniquely defined up to graph listings.  This is essentially explained in \cite{jy2} (Lemma 6.8), although that was stated for simply-connected graphs.

In the setting of the decomposition \eqref{g-over-e}, for $A \in \Mtos$, each object
\[
A(H) \defn \bigotimes_{v \in H} A\profv
\]
is well defined because each $H$ is unique up to graph listings and $A(H)$ only uses the vertex listings in $H$.

The following observation ensures that one can go from a well marked graph to a reduced marked graph uniquely.

\begin{proposition}
\label{unique-reduction}
Suppose:
\begin{itemize}
\item
$\cG$ is a shrinkable pasting scheme (Def. \ref{def:shrinkable}).
\item
$(G,\sds) \in \cG$ is a well marked graph.
\item
$E \subseteq \Edge(G)$ is the subset of all the internal edges with both end vertices normal.  
\end{itemize}
Then there is a unique weak isomorphism class of reduced marked graphs $[(G/E,\sds)]$ in which $G/E$ is obtained from $G$ by shrinking all $e \in E$.
\end{proposition}

\begin{proof}
The existence and uniqueness of $[(G/E,\sds)]$ was given in \eqref{g-over-e}.  It has the same set of distinguished vertices because, in shrinking the internal edges in $E$, no distinguished vertices are affected.  

Next observe that $(G/E,\sds)$ is a well marked graph.  Indeed, a flag $f$ in a distinguished vertex $w$ in $G/E$ must still be part of an internal edge because it was so in $G$ and it is not shrunk in forming $G/E$.  Moreover, the other end vertex of $f$ in $G$ is a normal vertex, which is either unaffected in passing to $G/E$ or is combined with some other normal vertices in $G$ to form a normal vertex in $G/E$.  In any case, the other end vertex of $f$ in $G/E$ is a normal vertex.

Finally, to see that $(G/E,\sds)$ is reduced, note that normal vertices in $G/E$ come from those in $G$ as discussed in the previous paragraph.  So $G/E$ cannot have any internal edge with both end vertices normal because $E$ is by definition the set of \emph{all} the internal edges in $G$ with both end vertices normal.
\end{proof}

\begin{corollary}
\label{graph-sub-reduced}
Suppose:
\begin{itemize}
\item
$\cG$ is a shrinkable pasting scheme, and $K \in \cG$ with $\Vt(K) \not= \varnothing$.
\item
For each $v \in \Vt(K)$, $(G_v, \sds_v) \in \cG$ is a well marked graph (e.g., reduced marked graph) such that $G_v$ has the same input/output profiles as $v$.  
\item
$\bigl(H = K(\{G_v\}),\sds_H\bigr)$ is the marked graph in Proposition \ref{well-marked-graphsub}.
\end{itemize}
Then there is a unique weak isomorphism class of reduced marked graphs $[(H',\sds_H)]$ in which $H'$ is obtained from $H$ by shrinking all the internal edges with both end vertices normal.
\end{corollary}

\begin{proof}
By Proposition \ref{well-marked-graphsub} $(H,\sds_H)$ is a well marked graph.  It has a unique reduction up to weak isomorphism by Proposition \ref{unique-reduction}.
\end{proof}

\section{Relative left properness}

In this section, we show that, if $\cG$ is a shrinkable pasting scheme admissible in $\M$ (Definition~\ref{gprop-model-operad}) and $\M$ is nice enough (Definition~\ref{def_especially_nice}), then the model category structure on $\gprop$ in Corollary~\ref{gprop-model-operad} satisfies a property close to that of left properness, which we will refer to as \textbf{relative} left properness.

Fix a  $\fC$-colored pasting scheme $\graphgpd$ for some non-empty set $\fC$ of colors and a bicomplete symmetric monoidal closed category $(\M, \otimes, \tensorunit)$.

\subsection{Vertex decoration}

\begin{definition}\label{def_ang}
Suppose $(G,\sds) \in \cG$ is a marked graph, and $A \in \M^{S}$ \eqref{m-tothe-s-again}. 
\begin{enumerate}
\item
For $u \in \vertex(G)$, write $A(u)$ for the entry of $A$ corresponding to the profiles of $u$.  In other words, if $u$ has profile $\dch \in S$, then
\[
A(u) = A\dc \in \M.
\]
\item
Define the object
\[
A\bigl(\sn(G)\bigr) 
= A\bigl(\sn(G,\sds)\bigr) 
= \bigotimes_{u \in \sn(G,\sds)} A(u) \in \M,
\]
where $\sn(G,\sds) = \vertex(G) \setminus \sds$ is the set of normal vertices.
\item
Write $[(G,\sds)]$ for the isomorphism class of $(G,\sds)$ in $\cG$.
\end{enumerate}
\end{definition}

\subsection{The pushout filtration}
If $H\to G$ is a homomorphism of groups, then the restriction functor
$\M^G \to \M^H$
has a left adjoint, called \textbf{induction}, 
\[
	G \underset{H}{\centerdot} (-) : \M^H \to \M^G. 
\]
Restriction and induction are actually a Quillen pair; see \cite{bm06} (2.5.1).
The following definition appears in \cite{em06} (Sec. 12) and \cite{harper-jpaa} (7.10).

\begin{definition}[$Q$-Construction]
\label{one-colored-q}
Suppose there is a map $i : X \to Y \in \calm$, and $0 \leq q \leq t$.  The object $Q^{t}_q = Q^t_q(i) \in \calm^{\Sigma_t}$ is defined as follows.
\begin{itemize}
\item
$Q^{t}_0 = X^{\otimes t}$.
\item
$Q^{t}_t = Y^{\otimes t}$.
\item
For $0 < q < t$ there is a pushout in $\calm^{\Sigma_t}$:
\begin{equation}
\label{inductive-q-one-colored}
\nicexy{
\Sigma_t
\underset{\Sigma_{t-q} \times \Sigma_q}{\centerdot} 
\left[ X^{\otimes (t-q)} 
\otimes Q_{q-1}^{q}\right] 
\ar[d]_-{(\id,i^{\boxprod q})} \ar[r] 
&
Q^{t}_{q-1} \ar[d]
\\
\Sigma_t
\underset{\Sigma_{t-q} \times \Sigma_q}{\centerdot} 
\left[ X^{\otimes (t-q)} 
\otimes 
Y^{\otimes q}\right]
\ar[r] 
&
Q^{t}_q.
}
\end{equation}
\end{itemize}
Write $i^{\boxprod t}$ for the natural map $Q^t_{t-1} \to Y^{\otimes t}$.  It is an iterated pushout product of $i$.
\end{definition}

Recall the definition of a shrinkable pasting scheme from Definition \ref{def:shrinkable}.  The following filtration is completely categorical and requires no model category structure.

\begin{lemma}
\label{opc-jt}
Suppose:
\begin{itemize}
\item
$\cG$ is a shrinkable pasting scheme, $A \in \gprop$, and
\item
$i : X \to Y \in \calm$, regarded as a map in $\M^S$ concentrated in the $s$-entry for some $s \in S$.
\item
The diagram
\begin{equation}
\label{gprop-pushout-free-map}
\nicexy{
\opg \comp X \ar[d]_{i_*} \ar[r]^-{f} 
& A \ar[d]^-{h}
\\
\opg \comp Y \ar[r]
& A_{\infty}
}
\end{equation}
is a pushout in $\gprop$.
\item
$[r] = \smallbinom{[\ud]}{[\uc]} \in S$ is an orbit.
\end{itemize}
Then the $[r]$-entry of the map $h$ is a countable composition
\begin{equation}
\label{pushout-filtration-gprop}
\nicexy{
A_0([r]) \ar[r]^-{h_1}
& A_1([r]) \ar[r]^-{h_2}
& A_2([r]) \ar[r]^-{h_3}
& \cdots \ar[r]
& \colim_k A_k([r]) \ar[d]^-{\cong}
\\
A([r]) \ar@{=}[u] &&&& A_\infty([r])
}
\end{equation}
in $\Mtor = \M^{\sigmabrcopsigmabrd}$, where for $k \geq 1$ the maps $h_k$ are inductively defined as the pushout 
\begin{equation}
\label{opg-hk-pushout}
\nicexy@R+10pt{
\coprod\limits_{[(G,\sds)]}
\sigmabrr \dotover{\Aut(G,\sds)} 
\Bigl\{\ang \otimes Q^k_{k-1}\Bigr\} 
\ar[d]_-{\amalg (\Id \otimes i^{\boxprod k})_*} 
\ar[r]^-{f^{k-1}_*} 
& 
A_{k-1}([r]) \ar[d]^-{h_k}
\\
\coprod\limits_{[(G,\sds)]}
\sigmabrr \dotover{\Aut(G,\sds)} 
\Bigl\{\ang \otimes Y^{\otimes k}\Bigr\} 
\ar[r]^-{\xi_{k}} 
& 
A_k([r])
}
\end{equation}
in $\Mtor = \M^{\sigmabrcopsigmabrd}$.  In this pushout:
\begin{enumerate}
\item
The coproducts are indexed by the set of weak isomorphism classes $[(G,\sds)]$ of \underline{reduced} marked graphs such that:
\begin{itemize}
\item
the input/output profile of $G$ is in the orbit $[r]$; 
\item
$\sds$ consists of $k$ vertices, all with profiles in the orbit $[s]$.
\end{itemize}
\item
The top horizontal map $f^{k-1}_*$ is induced by $f$ and the $\cG$-prop structure maps of $A$ \cite{jy2} (Lemma 10.40).
\item
$\sigmabrr = \sigmabrcopsigmabrd$.
\end{enumerate}
\end{lemma}

\begin{proof}
Define
\[
B([r]) = \colim_k A_k([r]) \in \Mtor.
\]
Corollary \ref{graph-sub-reduced} and the $\cG$-prop structure maps of $A$ imply that $B$ has a canonical $\cG$-prop structure together with a $\cG$-prop map $A \to B$ induced by $A_0 \to B$.  The map $Y \to B$ is induced by:
\begin{itemize}
\item
$\xi_1$ using the input and output extension \cite{jy2} (6.10 and 6.11) of the $s$-corolla whose only vertex is distinguished;
\item
the maps $\mathbb{1} \rightarrow A(c;c)$, where $c$ ranges over all colors of $A$;
\item
the natural map $A_1 \to B$.  
\end{itemize}
That $B$ is the pushout $A_\infty$ follows from its inductive definition.
\end{proof}

\begin{remark}
Suppose $C  = C_{(\uc;\ud)}$ is a corolla with the indicated profiles \cite{jy2} (1.31) and with unique vertex $w$.  Then its input and output extension mentioned in the previous proof is the spider graph:
\begin{center}
\begin{tikzpicture}
\matrix[row sep=-.1cm,column sep=.8cm] {
\node [plain] (x1) {$x_1$}; 
&& \node [plain] (xn) {$x_n$}; \\
& \node [plain, label=above:.. \small{$\ud$} .., label=below:.. \small{$\uc$} ..] (w) {$w$}; &\\
\node [plain] (v1) {$v_1$};
&& \node [plain] (vm) {$v_m$};\\
};
\draw [arrow] (v1) to (w);
\draw [arrow] (vm) to (w);
\draw [inputleg] (v1) to node[below=.2cm]{$c_1$} +(0,-.8cm);
\draw [inputleg] (vm) to node[below=.2cm]{$c_m$} +(0,-.8cm);
\draw [outputleg] (x1) to node[above=.2cm]{$d_1$} +(0,.8cm);
\draw [outputleg] (xn) to node[above=.2cm]{$d_n$} +(0,.8cm);
\draw [arrow] (w) to (x1);
\draw [arrow] (w) to (xn);
\end{tikzpicture}
\end{center}
It is obtained from $C$ by attaching to every leg a $1$-input, $1$-output corolla with equal flag colors.  In the previous proof, $w$ is distinguished with profiles $s = (\uc;\ud)$, and the newly attached vertices $v_i$'s and $ x_j$'s are all normal vertices to be used with the colored units of $A$.
\end{remark}

\begin{definition}
\label{def:sigma-cof}
Suppose $\M$ is a cofibrantly generated model category, and $\graphgpd$ is a pasting scheme.
\begin{enumerate}
\item
An object $A \in \M^S$ is called \textbf{$\sigmag$-cofibrant} in $\M$ if $A$ is cofibrant in $\Mtos$.
\item
A map in $\Mtos$ is called a \textbf{$\sigmag$-cofibration} in $\M$ if it is a cofibration in $\Mtos$.  
\item
If $C$ is a small category, a \textbf{$C$-cofibration} (resp., a \textbf{$C$-cofibrant object}) is a cofibration (resp., a cofibrant object) in the diagram category $\M^C$ with the projective model structure \cite{hirschhorn} (11.6.1).
\end{enumerate}
\end{definition}

\begin{remark}
Recall that there is a decomposition \eqref{m-tothe-s-again}
\[
\Mtos \cong 
\prod_{([\uc];[\ud]) \in S} \M^{\sigmabrcopsigmabrd}.
\]
So $\sigmag$-cofibrant / cofibration means the $[r]$-entry in $\Mtor$ is $\sigmabrr$-cofibrant / cofibration as $[r]$ runs through all the orbits in $S$. 
\end{remark}

\begin{definition}\label{def_especially_nice}
Let $\M$ be a cofibrantly generated monoidal model category and $\cG$ be a shrinkable pasting scheme. We say that $(\M,\cG)$ is a \textbf{compatible pair} if $\cG$ is admissible in $\M$ (Definition~\ref{gprop-model-operad}), every object of $\M$ is cofibrant and, for every well-marked, reduced $(G,\sds)$, the object $\ang$ is $\Aut(G,\sds)$-cofibrant whenever $A$ is $\Sigma_{\cG}$-cofibrant. 

Equivalently, we say that the model category $\M$ is compatible with $\cG$ or the pasting scheme $\cG$ is compatible with $\M$. 
\end{definition}

\begin{example}\label{example_especially_nice}
A sufficient condition for a model category $\M$ to be compatible with every shrinkable pasting scheme $\cG$ is that if $G$ is a finite group, then every object of $\M^{G}$ is cofibrant in the projective model structure. Examples of model categories satisfying this property include: 
\begin{itemize} 
\item The category of unbounded chain complexes $\textnormal{Ch}({k})$ over $k$ when $k$ is a field of characteristic zero with the projective model structure~\cite[2.3.11]{hovey}.
Every object of $\operatorname{Ch}(k[G])$ is cofibrant by Lemma \ref{lemma semisimple} since $k[G]$ is semisimple \cite[XVIII.1.2]{lang}.
In the same way, one sees that the category of non-negatively graded chain complexes satisfies this property.
\item The category of simplicial $k$-modules, again for $k$ of characteristic zero.
\item Quillen's categories of reduced rational simplicial (or $dg$) Lie algebras~\cite[II.5]{quillen}. 
\end{itemize}

Model category structures which are not compatible with all $\cG$ include the category of simplicial sets and simplicial abelian groups (See Section~\ref{appendix}). 
\end{example} 

\begin{remark} 
If $\M$ is a cofibrantly generated monoidal model category in which every object is cofibrant, then $\M$ is compatible   with the pasting scheme $\UTree$, for operads. In particular, $\M=\sset$ is compatible with $\UTree$ as shown in \cite[Definition 2.6.6; Theorem 3.1.10]{hryoperads}. 
\end{remark} 

\begin{remark}\label{remark_every_sigma_cof}

If $\M$ is a model category which satisfies the condition in Example~\ref{example_especially_nice}, then every object in $\gprop$ is $\sigmag$-cofibrant, but not necessarily cofibrant. In particular, when working in non-negatively graded chain complexes over a field of characteristic zero, every operad is $\sigmag$-cofibrant but there exist many examples which are not cofibrant, such as the associative operad $\mathbb{A}$ and the commutative operad $\mathbb{C}$. 

\end{remark}

\begin{proposition}
\label{sigmacof-closure}
Suppose that $(\M,\cG)$ is a compatible pair and that 
\begin{itemize}

\item
$i : X \to Y$ is a cofibration in $\calm$, regarded as a map in $\M^S$ concentrated at the $s$-entry for some $s \in S$.  
\item $A \in \gprop$ is $\sigmag$-cofibrant. 
\item
The diagram
\begin{equation}
\label{gprop-freepushout-of-a}
\nicexy{
\opg \comp X \ar[d]_{i_*} \ar[r]^-{f} 
& A \ar[d]^-{h}
\\
\opg \comp Y \ar[r]
& A_{\infty}
}
\end{equation}
is a pushout in $\gprop$.
\end{itemize}
Then:
\begin{enumerate}
\item
Each map
\[
\nicexy@C+.5cm{
\ang \otimes Q^k_{k-1} \ar[r]^-{\Id \otimes i^{\boxprod k}} 
& \ang \otimes Y^{\otimes k}}
\]
on the left side of the pushout \eqref{opg-hk-pushout} is an $\Aut(G,\sds)$-cofibration between $\Aut(G,\sds)$-cofibrant objects.
\item
Each map
\[
\nicexy@C+.5cm{
\sigmabrr \dotover{\Aut(G,\sds)} 
\Bigl\{\ang \otimes Q^k_{k-1}\Bigr\} 
\ar[r]^-{(\Id \otimes i^{\boxprod k})_*} 
& 
\sigmabrr \dotover{\Aut(G,\sds)} 
 \Bigl\{\ang \otimes Y^{\otimes k}\Bigr\} 
 }
\]
on the left side of the pushout \eqref{opg-hk-pushout} is a $\sigmabrr$-cofibration between $\sigmabrr$-cofibrant objects.
\item
The map
\[
\resizebox{\textwidth}{!}{
	\nicexy@C+.5cm{
	\coprod\limits_{[(G,\sds)]} 
	\sigmabrr \dotover{\Aut(G,\sds)} 
	\Bigl\{\ang \otimes Q^k_{k-1}\Bigr\} 
	\ar[r]^-{(\Id \otimes i^{\boxprod k})_*} 
	& 
	\coprod\limits_{[(G,\sds)]} 
	\sigmabrr \dotover{\Aut(G,\sds)} 
	\Bigl\{\ang \otimes Y^{\otimes k}\Bigr\} 
	}
}
\]
on the left side of the pushout \eqref{opg-hk-pushout} is a $\sigmabrr$-cofibration between $\sigmabrr$-cofibrant objects.
\item
The map $h_k : A_{k-1}([r]) \to A_k([r])$ on the right side of the pushout \eqref{opg-hk-pushout} is a $\sigmabrr$-cofibration between $\sigmabrr$-cofibrant objects.
\item
The map $h : A \to A_{\infty}$ is a $\sigmag$-cofibration between $\sigmag$-cofibrant $\cG$-props.
\end{enumerate}
\end{proposition}

\begin{proof}
For (1), $\ang$ is $\Aut(G,\sds)$-cofibrant by the assumption that $\M$ is compatible with $\cG$.  The pushout product axiom implies that the iterated pushout product $i^{\boxprod k} : Q^k_{k-1} \to Y^{\otimes k}$ is a cofibration between cofibrant objects in $\M$, since we are assuming that every object in $\M$ is cofibrant.  Moreover, $i^{\boxprod k}$ has an $\Aut(G,\sds)$-action because weak isomorphisms preserve distinguished vertices.  So Lemma 2.5.2 in \cite{bm06} implies that the map $\Id \otimes i^{\boxprod k}$ is an $\Aut(G,\sds)$-cofibration.

Furthermore, $\ang$ is $\Aut(G,\sds)$-cofibrant, and $Y^{\otimes k}$ is cofibrant in $\M$ and has an $\Aut(G,\sds)$-action.  So Lemma 2.5.2 in \cite{bm06} implies that   $\ang \otimes Y^{\otimes k}$ is $\Aut(G,\sds)$-cofibrant, and similarly $\ang \otimes Q^k_{k-1}$ is also $\Aut(G,\sds)$-cofibrant.

For (2), note that there is a left Quillen functor
\begin{equation}
\label{aut-to-sigma}
\nicexy@C+1.5cm{
\M^{\Aut(G,\sds)} \ar[r]^-{\sigmabrr \dotover{\Aut(G,\sds)} (?)}
& \Mtor},
\end{equation}
which is the left adjoint of the functor induced by restriction along $\Aut(G,\sds) \to \sigmabrr$.  Applying this left Quillen functor to the map $\Id \otimes i^{\boxprod k}$--which is an $\Aut(G,\sds)$-cofibration between $\Aut(G,\sds)$-cofibrant objects by (1)--yields a $\sigmabrr$-cofibration between $\sigmabrr$-cofibrant objects.

For (3), note that taking a coproduct of the maps in (2) still gives a $\sigmabrr$-cofibration between $\sigmabrr$-cofibrant objects by \cite{hirschhorn} (10.2.7 and 10.3.4).

For (4), the map $h_k$ is the pushout of the map in (3), so it is a $\sigmabrr$-cofibration.  An induction then shows that both its domain and codomain are $\sigmabrr$-cofibrant objects.

Assertion (5) follows from (4), that the orbit $[\ur]$ is arbitrary, and the fact that cofibrations are closed under transfinite compositions.
\end{proof}

The following observation says that the pushout of a weak equivalence between $\sigmag$-cofibrant $\cG$-props along a map that is the pushout of a free cofibration, is again a weak equivalence between $\sigmag$-cofibrant $\cG$-props.

\begin{proposition}
\label{left-proper-key}
Suppose that $(\M,\cG)$ is a compatible pair and that 
\begin{itemize}

\item
$i : X \to Y$ is a cofibration in $\calm$, regarded as a map in $\M^S$ concentrated at the $s$-entry for some $s \in S$.  
\item
 $f : A \to B \in \gprop$ is a weak equivalence between $\sigmag$-cofibrant $\cG$-props. 
\item
Both squares in the diagram
\begin{equation}
\label{pushout-b}
\nicexy{
\opg \comp X \ar[d]_{i_*} \ar[r] 
& A \ar[d]_-{h^A} \ar[r]^-{f}_-{\sim} 
& B \ar[d]^-{h^B}
\\
\opg \comp Y \ar[r]
& A_{\infty} \ar[r]^-{f_{\infty}}
& B_{\infty}
}
\end{equation}
in $\gprop$ are pushouts.  
\end{itemize}
Then $f_{\infty}$ is also a weak equivalence between $\sigmag$-cofibrant $\cG$-props.
\end{proposition}

\begin{proof}
Weak equivalences in $\gprop \cong \alg(\opg)$ are created entrywise in $\calm$.  The outer rectangle in \eqref{pushout-b} is also a pushout.  So each of the maps $h^A$ and $h^B$ is filtered, in which each $[r]$-entry, with $[r] \in S$ an arbitrary orbit, of the $k$-th map is a pushout as in \eqref{opg-hk-pushout}.   There is a commutative ladder diagram
\begin{equation}
\label{ladder-diagram-ab}
\nicexy@C-.3cm{
A([r]) \ar@{=}[r] \ar[d]_-{f}  &
A_0([r]) \ar[d]_-{f_0} \ar[r]^-{h_1^A}
& A_1([r]) \ar[d]_-{f_1} \ar[r]^-{h_2^A}
& \cdots \ar[r]
& \colim A_k([r]) \cong A_\infty([r]) \ar[d]^-{f_\infty}
\\
B([r]) \ar@{=}[r]  &
B_0([r]) \ar[r]^-{h_1^B}
& B_1([r]) \ar[r]^-{h_2^B}
& \cdots \ar[r]
& \colim B_k([r]) \cong B_\infty([r])
}
\end{equation}
in $\Mtor$.   All the horizontal maps $h_k^A$ and $h_k^B$ are cofibrations in $\Mtor$ by Proposition \ref{sigmacof-closure}, and so all the objects in the ladder diagram are cofibrant in $\Mtor$.  Using \cite{hirschhorn} (15.10.12(1)), in order to show that the map $f_\infty$ is a weak equivalence between cofibrant objects in $\Mtor$, it suffices to show that all the vertical maps $f_k$, with $0 \leq k < \infty$, are weak equivalences by induction on $k$.

The map $f_0$ is the $[r]$-entry of $f$, which is a weak equivalence by assumption.  Suppose $k \geq 1$. Consider the following commutative diagram in $\Mtor$.
\begin{equation}
\label{ab-pushout-cube}
\resizebox{\textwidth}{!}{
	\nicexy@C-1.3cm{
	\coprod
	\sigmabrr \dotover{\Aut(G,\sds)} 
	\Bigl\{\ang \otimes Q^k_{k-1}\Bigr\} 
	\ar[dd]_-{\amalg (\Id \otimes i^{\boxprod k})_*} 
	 \ar@(d,l)[dr]_-{f_*} \ar[rr] 
	&&  A_{k-1}([r])
	\ar[dr]^-{f_{k-1}} \ar'[d][dd] &
	\\
	& \coprod 
	\sigmabrr \dotover{\Aut(G,\sds)} 
	\Bigl\{\bng \otimes Q^k_{k-1}\Bigr\} 
	\ar[dd] \ar[rr] 
	&& B_{k-1}([r]) \ar[dd]
	\\
	\coprod 
	\sigmabrr \dotover{\Aut(G,\sds)} 
	\Bigl\{\ang \otimes Y^{\otimes k}\Bigr\} 
	\ar@(d,l)[dr]_-{f_{*}} \ar'[r][rr] 
	&& A_k([r])  \ar[dr]^-{f_{k}} &
	\\
	& \coprod 
	\sigmabrr \dotover{\Aut(G,\sds)} 
	\Bigl\{\bng \otimes Y^{\otimes k}\Bigr\} 
	\ar[rr] 
	&& B_k([r])
	}
}
\end{equation}
Both the back face (with $A$'s) and the front face (with $B$'s) are pushout squares as in \eqref{opg-hk-pushout},  and the maps from the back square to the front square are all induced by $f$.  The map $f_{k-1}$ is a weak equivalence by the induction hypothesis.  By Proposition \ref{sigmacof-closure}, all the objects in the diagram are cofibrant in $\Mtor$, and the left vertical maps in the back and the front faces are cofibrations in $\Mtor$.  So to show that the induced map $f_k$ is a weak equivalence, it is enough to show, by the Cube Lemma \cite{hovey} (5.2.6) / \cite{hirschhorn} (15.10.10), that both maps labeled as $f_*$ are weak equivalences.

To see that the top $f_*$ in the above diagram is a weak equivalence, note that a coproduct of weak equivalences between cofibrant objects is again a weak equivalence by Ken Brown's Lemma \cite{hovey} (1.1.12).  Using the left Quillen functor \eqref{aut-to-sigma} and Ken Brown's Lemma again, it is enough to show that, within each coproduct summand, the map
\begin{equation}
\label{angq-bngq}
\nicexy{\ang \otimes Q^k_{k-1} \ar[r]^{f_*} 
& \bng \otimes Q^k_{k-1}}
\end{equation}
is a weak equivalence between $\Aut(G,\sds)$-cofibrant objects.  By Proposition \ref{sigmacof-closure} the source and target of $f_*$ are $\Aut(G,\sds)$-cofibrant objects.  Recall that weak equivalences in any diagram category in $\M$ are defined entrywise.  The map
\[
\nicexy{\ang \ar[r]^-{f_*} & \bng}
\]
is a finite tensor product of entries of $f$, each of which is a weak equivalence between cofibrant objects in $\M$.  So this $f_*$ is a weak equivalence between cofibrant objects, and tensoring this map with the cofibrant object $Q^k_{k-1}$ yields a weak equivalence, since in any monoidal model category the tensor product is a left Quillen bifunctor.

A similar argument with $Y^{\otimes k}$ in place of $Q^k_{k-1}$ shows that the bottom $f_*$ in the commutative diagram is also a weak equivalence.  Therefore, as discussed above, $f_k$ is a weak equivalence, finishing the induction.
\end{proof}

The following is the main theorem of this section, and one of two main theorems of the paper.

\begin{theorem}
\label{gprop-relative-left-proper}
Suppose that $(\M,\cG)$ is a compatible pair. 
Then the cofibrantly generated model structure on $\gprop$ in Definition~\ref{gprop-model-operad} is left proper relative to $\sigmag$-cofibrant $\cG$-props, in the sense that pushouts along cofibrations preserve weak equivalences between $\sigmag$-cofibrant $\cG$-props.   
\end{theorem}

\begin{proof}
The set of generating cofibrations in $\gprop$ is $\opg \comp \sI$, where $\sI$ is the set of generating cofibrations in $\M^{|S|}$, each of which is concentrated in one entry and is a generating cofibration of $\calm$ there.  A general cofibration in $\gprop$ is a retract of a transfinite composition of pushouts of maps in $\opg \circ \sI$. So a retract and induction argument reduces the proof to the situation in Proposition \ref{left-proper-key}.
\end{proof}

\begin{corollary}
Under the assumptions of Example~\ref{example_especially_nice},  the cofibrantly generated model structure on $\gprop$ in Definition~\ref{gprop-model-operad} is left proper. 
\end{corollary}

\begin{proof} 
Every object of $\gprop$ is $\sigmag$-cofibrant by Remark~\ref{remark_every_sigma_cof}.
\end{proof}

The following observation says that cofibrant $\cG$-props are also $\sigmag$-cofibrant.  It will be used in Proposition \ref{chi-weq} below.

\begin{proposition}
\label{cof-is-sigmacof}
Suppose that $(\M,\cG)$ is a compatible pair and that $P$ is a cofibrant $\cG$-prop.
Then $P$ is also a $\sigmag$-cofibrant $\cG$-prop.
\end{proposition}

\begin{proof}
Recall that $\gprop$ is cofibrantly generated (Definition \ref{gprop-model-operad}).  Since $P \in \gprop$ is cofibrant, it is a retract of the colimit $P_{\infty}$ in a transfinite composition 
\begin{equation}
\label{cell-complex}
\nicexy{
\varnothing = P_0 \ar[r]^-{j_1} & P_1 \ar[r]^-{j_2} & \cdots \ar[r] & \colim_k P_k \defn P_{\infty}}
\end{equation}
in $\gprop$ starting at the initial $\cG$-prop $\varnothing$, in which each map $j_k : P_{k-1} \to P_k$ is a pushout as in Proposition \ref{sigmacof-closure} with $i : X \to Y$ a generating cofibration in $\M$.  To show that $P$ is $\sigmag$-cofibrant, it is enough to show that $P_{\infty}$ is $\sigmag$-cofibrant.  Moreover, the initial $\cG$-prop is $\sigmag$-cofibrant because, given our conditions on $\M$, the unit of $\M$ is cofibrant.\footnote{see \eqref{initial-wheelfree} and \eqref{initial-wheel} for an explicit description of the initial $\cG$-prop}.  Since cofibrations are closed under transfinite composition, it is enough to show that each map $j_k$ for $k \geq 1$ is a $\sigmag$-cofibration.  This is true by Proposition \ref{sigmacof-closure}(5) and an induction.
\end{proof}

\section{Derived change-of-base}\label{section basechange}

The main result in this section is Theorem \ref{derived-basechange}.  It says that, under certain assumptions, a Quillen equivalence between underlying categories induces a Quillen equivalence between the categories of $\cG$-props. 


Recall the definition of a pasting scheme being admissible in a monoidal model category $\M$ (Def. \ref{gprop-model-operad}) and the definition of a shrinkable pasting scheme (Def. \ref{def:shrinkable}).
We will also need the following definition, which appears in \cite{ss03} (Section 3.2).

\begin{definition}[Weak Monoidal Quillen Pairs] Suppose that $\M$ and $\N$ are monoidal model categories.
\begin{itemize}
	\item A \textbf{lax monoidal} structure on a functor $R : \N \rightarrow \M$ consists of a morphism $\nu: \tensorunit_{\M} \to R(\tensorunit_{\N})$, and natural morphisms \[ RX \otimes RY \to R(X \otimes Y)\] which are coherently associative and unital.
	\item A \textbf{weak monoidal Quillen pair} between monoidal model categories $\M$ and $\N$ consists of a Quillen adjoint pair
	\[ L : \M \adjoint \N : R\]
	with a lax monoidal structure on the right adjoint $R$ so that the following hold:
	\begin{itemize}
		\item For all cofibrant objects $A$ and $B$ in $\M$, the map $L(A\otimes B) \to LA \otimes LB$ (adjoint to $A \otimes B \to RLA \otimes RLB \to R(LA \otimes LB)$) is a weak equivalence in $\N$.
		\item For some cofibrant replacement $q: (\tensorunit_{\M})^c \overset\simeq\to \tensorunit_{\M}$, the composite map
		\[
			L(\tensorunit_{\M})^c \overset{Lq}\to L\tensorunit_{\M} \overset{\check \nu}\to \tensorunit_{\N}
		\]
		is a weak equivalence in $\N$.
	\end{itemize}
\end{itemize}

\end{definition}

Examples of weak monoidal Quillen pairs include the adjunction between reduced rational dg Lie algebras and reduced rational simplicial Lie algebras \cite{quillen}, and the Dold-Kan equivalence of chain complexes and simplicial abelian groups.

\begin{proposition}
\label{induced-adjoints}
Suppose:
\begin{itemize}
\item
$L : \M \adjoint \N : R$ is a weak symmetric monoidal Quillen pair with left adjoint $L$.
\item
$\cG = (S,\tg)$ is a pasting scheme which is admissible in $\M$.
\end{itemize}
Then there is an induced diagram with four Quillen pairs
\begin{equation}
\label{adjoint-diagram}
\nicexy{
\Mtos \ar@<2pt>[r]^-{L}  \ar@<2pt>[d]^-{\mathrm{free}}
& \Ntos \ar@<2pt>[d]^-{\mathrm{free}}  \ar@<2pt>[l]^-{R}
\\
\gpropm  \ar@<2pt>[r]^-{\Lg}  \ar@<2pt>[u]
& \gpropn  \ar@<2pt>[u]  \ar@<2pt>[l]^-{R}
}
\end{equation}
in which the following statements hold:
\begin{enumerate}
\item
The Quillen pair $(L,R)$ in the top row of \eqref{adjoint-diagram} is the entrywise prolongation of the original Quillen pair between $\M$ and $\N$.  This $(L,R)$ is a Quillen equivalence if the original adjoint pair between $\M$ and $\N$ is.
\item
Both $\gpropm$ and $\gpropn$ have the model structures in Definition \ref{gprop-model-operad}.
\item
Both vertical Quillen pairs are the free-forgetful adjunctions in \cite{jy2} (12.9), in which the undecorated right adjoints forget all of the $\cG$-prop structure except for the equivariant structure.
\item
At the bottom row the right adjoint $R$ is the entrywise prolongation of the original right adjoint as in \cite{jy2} (12.11(1)).
\item
The square of right adjoints commutes, and the square of left adjoints also commutes up to natural isomorphisms.
\item
If the original left adjoint $L$ is symmetric monoidal, then $\Lg$ is naturally isomorphic to the entrywise prolongation of $L$.
\end{enumerate}
\end{proposition}

\begin{proof}
The top and the vertical adjoint pairs exist as explained in the statements above.  The bottom horizontal left adjoint $\Lg$ exists by the Adjoint Lifting Theorem \cite{borceux} (4.5.6).  Every one of the four adjoint pairs is a Quillen pair because every right adjoint above preserves fibrations and acyclic fibrations, since they are defined entrywise in the underlying categories. 

If the original adjoint pair $(L,R)$ is a Quillen equivalence, then so is the top adjoint pair in \eqref{adjoint-diagram} by \cite{hirschhorn} (11.6.5(2)).

Furthermore, since both horizontal right adjoints are $R$ entrywise and both vertical right adjoints are forgetful functors, the right adjoints square commutes.  By uniqueness of left adjoints, the left adjoints square also commutes up to natural isomorphisms.

If the original left adjoint $L : \M \to \N$ is symmetric monoidal, then its entrywise prolongation is a functor $\gpropm \to \gpropn$ and is left adjoint to the entrywise prolongation of $R$ by \cite{jy2} (12.13).  So there is a natural isomorphism $\Lg \cong L$ by uniqueness of left adjoints.
\end{proof}

The following definition is a way of measuring how different $\Lg$ is from $L$ when the latter is not symmetric monoidal, but only weakly symmetric monoidal.

\begin{definition}
Suppose:
\begin{itemize}
\item
$L : \M \adjoint \N : R$ is a weak symmetric monoidal Quillen pair with left adjoint $L$.
\item
$\cG = (S,\tg)$ is a (not necessarily shrinkable) pasting scheme which is admissible in both $\M$ and $\N$.
\item 
$P \in \gpropm$.
\end{itemize}
Denote by
\begin{equation}
\label{transition-map}
\nicexy{
LP \ar[r]^-{\chi_P} & \Lg P \in \Ntos}
\end{equation}
the adjoint of the unit map $P \to R\Lg P$ regarded in $\Mtos$.
\end{definition}

\begin{remark}
For simplicity we omitted all the forgetful functors in the map $\chi_P$.  Denoting by $U$ the forgetful functors, the map $\chi_P$ is $LUP \to U\Lg P$.
\end{remark}

The following observation says that for the initial $\cG$-prop, $\Lg$ and $L$ are not all that different.  It will serve as the initial case in the induction in the proof of Proposition \ref{cof-chi-weq} below.

\begin{proposition}
\label{initialprop-chi}
Suppose:
\begin{itemize}
\item
$L : \M \adjoint \N : R$ is a weak symmetric monoidal Quillen pair with left adjoint $L$.

\item
Both $\M$ and $\N$ are compatible with $\cG$.
\item $P_0$ is the initial $\cG$-prop in $\M$.
\end{itemize}
Then the map $\chi_{P_0} : LP_0 \to \Lg P_0$ is a weak equivalence.
\end{proposition}

\begin{proof}
Since $P_0$ is the initial $\cG$-prop in $\M$, $\Lg P_0$ is the initial $\cG$-prop in $\N$ because $\Lg$ is a left adjoint.  There are now two cases.
\begin{enumerate}
\item
Suppose $\cG$ is wheel-free, i.e., $\cG$ is properly contained in $\gupc$.  Then
\begin{equation}
\label{initial-wheelfree}
P_0\dc = \begin{cases}
\tensorunit_{\M} & \text{if  $\dc = \ccsingle$},\\
\varnothing & \text{otherwise},
\end{cases}
\end{equation}
and similarly for $\Lg P_0$.
For a color $c \in \fC$ for which $\uparrow_c~ \in \cG$, the map $\chi_{P_0}$ at the $\ccsingle$-entry is the counit map $L\tensorunit_{\M} \to \tensorunit_{\N}$, which is a weak equivalence as part of the definition of a weak symmetric monoidal Quillen pair because $\tensorunit_{\M}$ is cofibrant.  In any other entry, the map $\chi_{P_0}$ is the unique self map of the initial object in $\N$.
\item
If $\cG$ has wheels, then
\begin{equation}
\label{initial-wheel}
P_0\dc = 
\begin{cases}
\tensorunit_{\M} & \text{ if $\dc = \ccsingle$},\\
\coprod\limits_{c \in \fC} 
\tensorunit_{\M} & \text{ if $\dc = \emptyprof$},\\
\varnothing & \text{otherwise},
\end{cases}
\end{equation}
and similarly for $\Lg P_0$.  As in the wheel-free case, the map $\chi_{P_0}$ is either a weak equivalence or the identity map of the initial object in entries $\dc \not= \emptyprof$.  At the $\emptyprof$-entry, $\chi_{P_0}$ is the map
\[
L\Bigl(\coprod \tensorunit_{\M}\Bigr) \cong
\coprod L\tensorunit_{\M} \to 
\coprod \tensorunit_{\N},
\]
which is a coproduct of the counit map $L\tensorunit_{\M} \to \tensorunit_{\N}$.  The counit map is a weak equivalence between cofibrant objects in $\N$, hence so is the coproduct.
\end{enumerate}
This shows that the map $\chi_{P_0}$ is a weak equivalence.
\end{proof}

Roughly the following observation says that, if the map $\chi_A$ is a weak equivalence, then it remains so after we attach a free cell to $A$, provided $A$ is cofibrant.   It will serve as the induction step in the proof of Proposition \ref{cof-chi-weq} below.

\begin{proposition}
\label{chi-weq}
Suppose:
\begin{itemize}
\item
$L : \M \adjoint \N : R$ is a weak symmetric monoidal Quillen pair with left adjoint $L$.

\item
Both $\M$ and $\N$ are compatible with $\cG$. 
\item
$i : X \to Y$ is a cofibration in $\calm$, regarded as a map in $\M^S$ concentrated at the $s$-entry for some $s \in S$.  
\item
The diagram
\begin{equation}
\label{gprop-pushout-freecof}
\nicexy{
\opg \comp X \ar[d]_{i_*} \ar[r]^-{f} 
& A \ar[d]^-{h}
\\
\opg \comp Y \ar[r]
& A_{\infty}
}
\end{equation}
is a pushout in $\gprop$ with $A$ cofibrant.
\item
The map $\chi_A : LA \to \Lg A$ is a weak equivalence.
\end{itemize}
Then the map $\chi_{A_\infty} : LA_{\infty} \to \Lg A_{\infty}$ is also a weak equivalence.
\end{proposition}

\begin{proof}
Pick an orbit $[r] \in S$, and consider the filtration \eqref{pushout-filtration-gprop}  in $\M^{\sigmabrr}$
\[
\nicexy{
A([r]) = A_0([r]) \ar[r]^-{h_1}
& A_1([r]) \ar[r]^-{h_2}
& \cdots \ar[r]
& \colim_k A_k([r]) 
\cong A_{\infty}([r])
}\]
of the $[r]$-entry of the map $h : A \to A_{\infty}$.  Applying the left adjoint $L$ in the $[r$]-entry, we obtain the filtration in $\N^{\sigmabrr}$
\[
\nicexy@C-.4cm{
L(A[r]) = L(A_0[r]) \ar[r]^-{Lh_1}
& L( A_1[r]) \ar[r]^-{Lh_2}
& \cdots \ar[r]
& \colim_k L(A_k[r])
\cong L(A_{\infty}[r])
}\]
of the map $(Lh)[r]$.

On the other hand, applying the left adjoint $\Lg$ to the pushout diagram \eqref{gprop-pushout-freecof} in $\gpropm$, we obtain the pushout diagram
\begin{equation}
\label{lga-freepushout}
\nicexy{
\opg \comp LX \ar[d]_{(Li)_*} \ar[r] 
& \Lg A \defn B \ar[d]^-{\Lg h}
\\
\opg \comp LY \ar[r]
& \Lg A_{\infty} \defn B_\infty
}
\end{equation}
in $\gpropn$.  To obtain the left side, we used the fact that the left adjoints square in \eqref{adjoint-diagram} commutes up to natural isomorphism.  Applying Lemma \ref{opc-jt} to this pushout diagram, we obtain a corresponding bottom horizontal filtration and a commutative ladder diagram
\begin{equation}
\label{ladder-la}
\nicexy@C-.4cm{
L(A[r]) = L(A_0[r]) \ar[r] \ar[d]_-{\xi_0}^-{= \chi_A}
& L( A_1[r]) \ar[r] \ar[d]_-{\xi_1}
& \cdots \ar[r]
& L(A_{\infty}[r]) \ar[d]_-{\xi_\infty}^-{= \chi_{A_\infty}}
\\
(\Lg A)[r] = B[r] = B_0[r] \ar[r] & B_1[r] \ar[r] & \cdots \ar[r] & B_\infty[r] = (\Lg A_\infty)[r]
}
\end{equation}
in $\N^{\sigmabrr}$.

Since $A \in \gpropm$ is cofibrant, it is also $\sigmag$-cofibrant in $\M$ by Proposition \ref{cof-is-sigmacof}, and hence $LA$ is $\sigmag$-cofibrant in $\N$ because $L$ is a left Quillen functor.   Likewise, $\Lg A \in \gpropn$ is cofibrant because $\Lg$ is a left Quillen functor by Proposition \ref{induced-adjoints}, so it is also $\sigmag$-cofibrant in $\N$ (Proposition \ref{cof-is-sigmacof}).  Moreover, all the horizontal maps are $\sigmabrr$-cofibrations by Proposition \ref{sigmacof-closure}.  Using once again \cite{hirschhorn} (15.10.12(1)), to show that the map $\chi_{A_\infty}$ is a weak equivalence, it is enough to show that all the vertical maps $\xi_k$ in \eqref{ladder-la} for $0 \leq k < \infty$ are weak equivalences by induction.  The map $\xi_0$ is the $[r]$-entry of the map $\chi_A$, so it is a weak equivalence.

For the induction step, suppose $\xi_{k-1}$ is a weak equivalence.  We must show that $\xi_k$ is a weak equivalence, for which we will use a cube argument similar to \eqref{ab-pushout-cube}.  Consider the commutative diagram in $\Ntor$:
\begin{equation}
\label{la-pushout-cube}
\resizebox{\textwidth}{!}{
	\nicexy@C-1.5cm{
	\coprod 
	\sigmabrr \dotover{\Aut(G,\sds)} 
	L\Bigl\{\ang \otimes Q^k_{k-1}(i)\Bigr\} 
	\ar[dd]_-{\amalg (\Id \otimes i^{\boxprod k})_*} 
	 \ar@(d,l)[dr]_-{\alpha} \ar[rr] 
	&&  L(A_{k-1}[r])
	\ar[dr]^-{\xi_{k-1}} \ar'[d][dd] &
	\\
	& \coprod 
	\sigmabrr \dotover{\Aut(G,\sds)} 
	\Bigl\{\bng \otimes Q^k_{k-1}(Li)\Bigr\} 
	\ar[dd]_(.7){\amalg (\Id \otimes (Li)^{\boxprod k})_*}  \ar[rr]
	&& B_{k-1}([r]) \ar[dd]
	\\
	\coprod 
	\sigmabrr \dotover{\Aut(G,\sds)} 
	L\Bigl\{\ang \otimes Y^{\otimes k}\Bigr\} 
	\ar@(d,l)[dr]_-{\beta} \ar'[r][rr] 
	&& L(A_k[r])  \ar[dr]^-{\xi_{k}} &
	\\
	& \coprod 
	\sigmabrr \dotover{\Aut(G,\sds)} 
	\Bigl\{\bng \otimes (LY)^{\otimes k}\Bigr\} 
	\ar[rr] 
	&& B_k([r])
	}
}
\end{equation}
The back face is $L$ applied to the pushout square \eqref{opg-hk-pushout} corresponding to the given pushout square \eqref{gprop-pushout-freecof}.  The front face is the  pushout square \eqref{opg-hk-pushout} corresponding to the pushout square \eqref{lga-freepushout}.  The objects $Q^k_{k-1}(i)$ and $Q^k_{k-1}(Li)$ refer to the $Q$-construction (Def. \ref{one-colored-q}) for the maps $i$ and $Li$, respectively.  By Proposition \ref{sigmacof-closure}, the left vertical maps in the back and the front faces are $\sigmabrr$-cofibrations in $\N$, and all the objects are $\sigmabrr$-cofibrant.  Thus, by the Cube Lemma \cite{hovey} (5.2.6), to show that $\xi_k$ is a weak equivalence, it is enough to show that the maps $\alpha$ and $\beta$ are weak equivalences.

The map $\alpha$ is a coproduct of the following compositions in $\Ntor$:
\[
\resizebox{\textwidth}{!}{
	\nicexy{
	\sigmabrr \dotover{\Aut(G,\sds)} 
	L\Bigl\{\ang \otimes Q^k_{k-1}(i)\Bigr\} 
	\ar[r]^-{\alpha_1}
	&
	\sigmabrr \dotover{\Aut(G,\sds)} 
	\Bigl\{\lang \otimes Q^k_{k-1}(Li)\Bigr\} 
	\ar[d]_-{\alpha_2}
	\\
	&\sigmabrr \dotover{\Aut(G,\sds)} 
	\Bigl\{\bng \otimes Q^k_{k-1}(Li)\Bigr\} 
	}
}
\]
The map $\alpha_1$ is induced by the lax comonoidal structure map of $L$ \cite{ss03} (3.4), while the map $\alpha_2$ is induced by the map $\chi_A$.  Since all three objects are $\sigmabrr$-cofibrant by Proposition \ref{sigmacof-closure}, to show that $\alpha$ is a weak equivalence, it is enough to show that $\alpha_1$ and $\alpha_2$ are weak equivalences.

Since $\chi_A$ is a weak equivalence, the same argument as in the paragraph containing \eqref{angq-bngq} shows that $\alpha_2$ is a weak equivalence.  For $\alpha_1$, first note that by Proposition \ref{sigmacof-closure} and the fact that $L$ is a left Quillen functor, the domain and the codomain of the map
\[
\nicexy{
L\Bigl\{\ang \otimes Q^k_{k-1}(i)\Bigr\} 
\ar[r]^{\alphabar_1} &
\lang \otimes Q^k_{k-1}(Li)
}\]
are $\Aut(G,\sds)$-cofibrant in $\N$.   Using the left Quillen functor \eqref{aut-to-sigma} and Ken Brown's Lemma \cite{hovey} (1.1.12), to show that $\alpha_1$ is a weak equivalence, it is enough to show that $\alphabar_1$ is a weak equivalence.  Since $(L,R)$ is a weak monoidal Quillen pair and all the objects in $\M$ and $\N$ are cofibrant, $\alphabar_1$ is a weak equivalence by \cite{muro14} (4.3).  This finishes the proof that $\alpha$ is a weak equivalence.  The proof that $\beta$ is a weak equivalence is the same as the one above for $\alpha$ but with $Y^{\otimes k}$ and $(LY)^{\otimes k}$ replacing $Q^k_{k-1}(i)$ and $Q^k_{k-1}(Li)$, respectively.  This finishes the induction step that $\xi_k$ is a weak equivalence.
\end{proof}


The following key observation is the generalized prop version of \cite{ss03} (5.1(1), for monoids) and \cite{muro14} (4.2, for $1$-colored non-symmetric operads).  Note that those two settings are both non-symmetric, while most shrinkable pasting schemes (Proposition \ref{shrinkable-pasting-schemes}) have graphs that encode symmetric group actions.  In particular, the category $\Ntos$ \eqref{m-tothe-s-again}, where the map $\chi_P$ lives, captures the underlying equivariant structure.

\begin{proposition}
\label{cof-chi-weq}
Suppose:
\begin{itemize}
\item
$L : \M \adjoint \N : R$ is a weak symmetric monoidal Quillen pair with left adjoint $L$.
\item
Both $\M$ and $\N$ are compatible with $\cG$.
\item 
$P$ is a cofibrant $\cG$-prop in $\M$.
\end{itemize}
Then the map $\chi_P : LP \to \Lg P$ is a weak equivalence in $\Ntos$.
\end{proposition}

\begin{proof}
Exactly as in \eqref{cell-complex}, $P$ is a retract of the colimit $P_\infty$ of a transfinite composition
\[
\nicexy{
\varnothing = P_0 \ar[r] & P_1 \ar[r] & \cdots \ar[r] & \colim_k P_k = P_{\infty}
}\]
in $\gpropm$ starting at the initial $\cG$-prop $P_0 = \varnothing$, where each map $P_{k-1} \to P_k$ is a pushout as in \eqref{gprop-pushout-free-map} with $i : X \to Y$ a generating cofibration in $\M$.  By the naturality of $\chi_{(-)}$, it is enough to show that $\chi_{P_{\infty}}$ is a weak equivalence in $\Ntos$.

Applying $L$ and $\Lg$ to the transfinite composition, we obtain a commutative ladder diagram
\[
\nicexy{
LP_0 \ar[d]_-{\chi_{P_0}} \ar[r]
& LP_1 \ar[d]_-{\chi_{P_1}} \ar[r]
& \cdots \ar[r]
& \colim_k LP_k \cong LP_\infty \ar[d]_-{\chi_{P_\infty}}
\\
\Lg P_0 \ar[r]
& \Lg P_1 \ar[r]
& \cdots \ar[r]
& \colim_k \Lg P_k \cong \Lg P_\infty
}\]
in $\Ntos$.  The initial $\cG$-prop $P_0$ is $\sigmag$-cofibrant in $\M$, and $LP_0$ is $\sigmag$-cofibrant in $\N$.  Since $\Lg P_0$ is the initial $\cG$-prop in $\N$, it is likewise $\sigmag$-cofibrant.  By Propositions \ref{sigmacof-closure} and the fact that $L$ and $\Lg$ are left Quillen functors, all the horizontal maps in the ladder diagram are $\sigmag$-cofibrations.  So in particular all the objects are $\sigmag$-cofibrant.  Using \cite{hirschhorn} (15.10.12(1)) and the ladder diagram, to show that $\chi_{P_\infty}$ is a weak equivalence, it is enough to show by induction that each $\chi_{P_k}$ for $k \geq 0$ is a weak equivalence.   

Since $P_0$ is the initial $\cG$-prop, the map $\chi_{P_0}$ is a weak equivalence by Proposition \ref{initialprop-chi}.  For the induction step, suppose $\chi_{P_{k-1}} : LP_{k-1} \to \Lg P_{k-1}$ is a weak equivalence.  We must show that $\chi_{P_k} : LP_k \to \Lg P_k$ is a weak equivalence.  Note that all the $P_j$'s are cofibrant in $\gpropm$ because $P_0$ is cofibrant and each map $P_{j-1} \to P_j$ is a pushout of a generating cofibration, hence itself a cofibration.  Therefore, $\chi_{P_k}$ is a weak equivalence by Proposition \ref{chi-weq}.
\end{proof}

The following observation is our main result about derived change-of-base categories.  It is the $\cG$-prop version of \cite{ss03} (3.12(3), for monoids) and \cite{muro14} (1.1, for $1$-colored non-symmetric operads).

\begin{theorem}
\label{derived-basechange}
Suppose:
\begin{itemize}
\item
$L : \M \adjoint \N : R$ is a weak symmetric monoidal Quillen equivalence with left adjoint $L$.

\item
Both $\M$ and $\N$ are compatible with $\cG$.
\end{itemize}
Then there is a Quillen equivalence
\[
\nicexy{
\gpropm \ar@<2pt>[r]^-{\Lg}
& \gpropn \ar@<2pt>[l]^-{R}
}\]
with left adjoint $\Lg$.
\end{theorem}

\begin{proof}
This is a Quillen adjunction by Proposition \ref{induced-adjoints}.  To show that it is a Quillen equivalence, suppose  $\varphi : \Lg A \to B$ is a map with $A \in \gpropm$  cofibrant and $B \in \gpropn$ fibrant.  We must show that $\varphi$ is a weak equivalence if and only if its adjoint map $A \to RB$ is a weak equivalence.  Recall that weak equivalences are defined entrywise in the underlying categories.

Consider the composition
\[
\nicexy{
LA \ar[r]^-{\chi_A} & \Lg A \ar[r]^-{\varphi} & B}
\]
in $\Ntos$, where $\chi_A$ is the map defined in \eqref{transition-map}.  Since $A$ is cofibrant, the map $\chi_A$ is a weak equivalence by Proposition \ref{cof-chi-weq}.  Thus, $\varphi$ is a weak equivalence if and only if $\varphi\chi_A$ is a weak equivalence in $\Ntos$.

The $\cG$-prop $A$ is cofibrant in $\Mtos$ by Proposition \ref{cof-is-sigmacof}, and $B$ is fibrant in $\Ntos$ because fibrations in $\Ntos$ and $\gpropn$ are both defined entrywise in $\N$.  Since $(L,R)$ is a Quillen equivalence, the entrywise prolongations $\Mtos \adjoint \Ntos$ in \eqref{adjoint-diagram} form a Quillen equivalence.  So $\varphi\chi_A$ is a weak equivalence if and only if its adjoint $A \to RB \in \Ntos$ is a weak equivalence, which finishes the proof.
\end{proof}

\section{Obstructions to (relative) left properness}\label{appendix}

The restrictions we put on our model categories in Definition~\ref{def_especially_nice} seem rather mysterious at first glance. This final section is intended to illuminate as much as the authors understand about these conditions. When attempting to prove that a category of generalized props is left proper, one is forced to analyze of pushouts of graphs which are decorated with objects of the base category. When these graphs are then assembled to give a $\gprop$-structure, as in Lemma~\ref{opc-jt}, we are tensoring objects from the base category together in such a way that it respects graph structure. If we want to impose that the assembled prop satisfies some additional property this will be a complicated interplay between the objects of the model category $\M$ and the graph automorphisms in the chosen pasting scheme.

\begin{proposition}\label{prop_counter_example}
Let $\sset$ be the category of simplicial sets with the Kan model category structure and let $\sset^{\Sigma_{n}}$ be the category of simplicial sets with $\Sigma_n$ action, with the projective model structure. 
The map
\[
	\bigotimes \colon \sset^{\times n} \to \sset^{\Sigma_n}
\]
does not preserve cofibrant objects. 
\end{proposition} 

To see this, let $A$ be any nonempty simplicial set.  We know that $A$ is cofibrant in $\sset$ and that the object $(A, A)$ is cofibrant in $\sset^{\times 2}$. We would like to thank Sean Tilson for showing us a proof of the following lemma.

\begin{lemma}
$A\times A$ is not cofibrant as an object of $\sset^{\Sigma_2}$. 
\end{lemma} 

\begin{proof} 
Let $\bigcup_{n\geq0} S^{n} = S^{\infty}\subseteq \mathbb{R}^{\infty}$ be the infinite sphere and let 
\[
	X = S_\bullet (S^\infty) \subset S_{\bullet}(\mathbb R^\infty)
\] 
be the singular $\Sigma_2$ simplicial set where the $\Sigma_2$ action is given by multiplication by $-1$. Note that $X$ has no fixed points and so $X^{\Sigma_2} = \varnothing$. Now, there exists a unique diagram
\[ \begin{tikzcd}
\varnothing \dar \rar & X \dar \\
A \times A \rar & *.
\end{tikzcd} \]
where all maps are $\Sigma_2$-equivariant maps. The map on the right is an acyclic fibration in $\sset$.
But there cannot be a $\Sigma_2$-equivariant lift $q: A\times A \to X$ in this diagram, since if there were there would be a factorization
\[\includegraphics{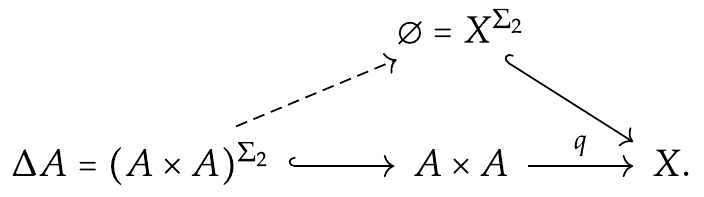}\]
Such a factorization does not exist since $\Delta A \cong A \neq \varnothing$. Hence $\varnothing \to A \times A$ is not a cofibration in $\sset^{\Sigma_2}$. 
\end{proof}

Notice that if replace $\sset$ with $\textnormal{Ch}({\mathbb Q})$ the map
\[
	\bigotimes \colon \textnormal{Ch}({\mathbb Q})^{\times n} \to \textnormal{Ch}({\mathbb Q})^{\Sigma_n}
\]
preserves cofibrant objects. This is because every module over $\mathbb{Q}[\Sigma_n]$ is projective, so every object in $\textnormal{Ch}({\mathbb Q})^{\Sigma_n}$ is cofibrant.

The following example shows obstructions to compatibility for wheeled properads in $\sset$.

\begin{example} 

Let $\cG=\gwheelc (\fC)$ be the pasting scheme (Definition~\ref{def:shrinkable}) of connected wheeled graphs (for wheeled properads), and consider the legless graph $G$ which has two normal vertices and one distinguished vertex.

\[
\begin{tikzpicture} [scale=.5]

\node[shape=circle, draw=white](q) at (-5,0){$G=$};
\node[shape=circle,draw=black] (d) at (0,0) {$d$};
\node[shape=circle, draw=black](n) at (-3,2) {$n_1$};
\node[shape=circle, draw=black](y) at (3,2) {$n_2$};

\draw[->](d) to  (n);
\draw[->] (d) to (y);

\end{tikzpicture}
\]

In this case, if we decorate $G$ by an object $A\in\M^{S}$ (Definition~\ref{def_ang}), we will end up with a reduced wheeled graph in which $\Aut(G,\sds)$ permutes the vertices $n_1$ and $n_2$, but, as we showed in Propostion~\ref{prop_counter_example}, the tensor product $A(n_1)\otimes A(n_2)$ will not necessarily be cofibrant. 

\end{example}



\begin{thebibliography}{AAAAAA}

\bibitem[BB14]{bb14} 
{\scshape Batanin, M.A.; Berger, C.} Homotopy theory for algebras over polynomial monads. Preprint, \href{https://arxiv.org/abs/1305.0086v6}{arXiv:1305.0086v6} [math.CT]. 
 



\bibitem[BM06]{bm06}
{\scshape Berger, C.; Moerdijk, I.} The Boardman-Vogt resolution of operads in monoidal model categories. {\em Topology} {\bf 45} (2006) 807--849. \mrev{2248514} (2008e:18016), \zbl{1105.18007}.


\bibitem[BM07]{bm07}
{\scshape Berger, C.; Moerdijk, I.} Resolution of coloured operads and rectification of homotopy algebras. {\em Contemp. Math.} {\bf 431} (2007) 31--58. \mrev{2342815} (2008k:18008), \zbl{1134.18005}.






\bibitem[Bor94]{borceux}
{\scshape Borceux, F.} Handbook of Categorical Algebra 2:  Categories and Structures. {\em Cambridge Univ. Press, Cambridge, UK,} 1994. \mrev{1313497} (96g:18001b), \zbl{0843.18001}.









\bibitem[DK80]{dk} 
{\scshape Dwyer, W.G.; Kan D.M.} Simplicial localizations of categories. {\em J. Pure Appl. Algebra} {\bf 17}(3) (1980) 267--284. \mrev{0579087} (81h:55018), \zbl{0485.18012}.




\bibitem[EM06]{em06}
{\scshape Elmendorf, A.D.; Mandell, M.A.} Rings, modules, and algebras in infinite loop space theory.  {\em Adv. Math.} {\bf 205}(1) (2006) 163-228. \mrev{2254311} (2007g:19001), \zbl{1117.19001}.




\bibitem[Fre09]{fresse_book} 
{\scshape Fresse, B.} Modules over operads and functors.  Lect. Notes in Math., 1967. {\em Springer-Verlag, Berlin,} 2009. \mrev{2494775} (2010e:18009), \zbl{1178.18007}.

\bibitem[Fre10]{fresse}
{\scshape Fresse, B.} Props in model categories and homotopy invariance of structures. {\em Georgian Math. J.} {\bf 17} (2010), 79--160. \mrev{2640648} (2011h:18011), \zbl{1227.18007}.

\bibitem[Gan04]{gan}
{\scshape Gan, W.L.} Koszul duality for dioperads. {\em Math. Res. Lett.} {\bf 10} (2003), 109-124. \mrev{1960128} (2004a:18005), \zbl{1103.18010}.





\bibitem[HR]{hr2}
{\scshape Hackney, P.; Robertson, M.} The homotopy theory of simplicial props. To appear in Israel Journal of Mathematics, \href{https://arxiv.org/abs/1209.1087}{arXiv:1209.1087} [math.AT]. 



\bibitem[HRY16]{hryoperads}
{\scshape Hackney, P.; Robertson, M.; Yau, D.} Relative left properness of colored operads. {\em Algebr. Geom. Topol.} {\bf 16} (2016), 2691--2714. \mrev{3572345}, \zbl{1350.18016}.

\bibitem[Har10]{harper-jpaa}
{\scshape Harper, J.E.} Homotopy theory of modules over operads and non-$\Sigma$ operads in monoidal model categories. {\em J. Pure Appl. Algebra} {\bf 214} (2010), 1407--1434. \mrev{2593672} (2011g:55024), \zbl{1231.55011}.





\bibitem[Hir03]{hirschhorn}
{\scshape Hirschhorn, P.S.} Model categories and their localizations. Math. Surveys and Monographs 99. {\em Amer. Math. Soc. Providence, RI,} 2003. \mrev{1944041} (2003j:18018), \zbl{1017.55001}.


\bibitem[Hov99]{hovey}
{\scshape Hovey, M.} Model categories. Math. Surveys and Monographs 63. {\em Amer. Math. Soc. Providence, RI,} 1999. \mrev{1650134} (99h:55031), \zbl{0909.55001}.


\bibitem[JY09]{jy1}
{\scshape Johnson, M.W.; Yau, D.} On homotopy invariance for algebras over colored PROPs. {\em J. Homotopy and Related Structures} {\bf 4} (2009), 275--315. \mrev{2559644} (2010j:18014), \zbl{1188.18007}.






\bibitem[Lan02]{lang} {\scshape Lang, S.} Algebra. Revised third edition. Grad. Texts in Math., 211. {\em Springer-Verlag, New York,} 2002. \mrev{1878556} (2003e:00003), \zbl{0984.00001}.


\bibitem[Lur09]{lurie}
{\scshape Lurie, J.} Higher Topos Theory. Annals of Math. Studies, 170. {\em Princeton Univ. Press, Princeton, NJ,} 2009. \mrev{2522659} (2010j:18001), \zbl{1175.18001}.


\bibitem[Mac98]{maclane}
{\scshape Mac Lane, S.} Categories for the working mathematician. Second edition. Grad. Texts in Math., 5. {\em Springer-Verlag, New York,} 1998. \mrev{1712872} (2001j:18001), \zbl{0906.18001}.


\bibitem[MMS09]{mms}
{\scshape Markl, M.; Merkulov, S.A.; Shadrin, S.} Wheeled PROPs, graph complexes and the master equation. {\em J. Pure Appl. Alg.} {\bf 213} (2009) 496-535. \mrev{2483835} (2010e:18010), \zbl{1175.18002}.


\bibitem[MSS02]{mss}
{\scshape Markl, M.; Shnider, S.; Stasheff, J.} Operads in Algebra, Topology and Physics. Math. Surveys and Monographs, 96. {\em Amer. Math. Soc., Providence, RI,} 2002. \mrev{1898414} (2003f:18011), \zbl{1017.18001}.





\bibitem[Mer10]{merkulov3}
{\scshape Merkulov, S.A.} Wheeled props in algebra, geometry and quantization. {\em European Congress of Math.,} 83--114. {\em Eur. Math. Soc., Z\"{u}rich,} 2010. \mrev{2648322} (2011i:18018), \zbl{1207.18010}.

\bibitem[Mur11]{muro11}
{\scshape Muro, F.} Homotopy theory of nonsymmetric operads. {\em Alg. Geom. Top.} {\bf 11} (2011) 1541--1599. \mrev{2821434} (2012k:18015), \zbl{1228.18006}.


\bibitem[Mur14]{muro14}
{\scshape Muro, F.} Homotopy theory of non-symmetric operads, II: Change of base category and left properness. {\em Alg. Geom. Top.} {\bf 14} (2014) 229--281; corrections, \href{https://arxiv.org/abs/1507.06644}{arXiv:1507.06644}. \mrev{3158759}, \zbl{1281.18001}.



\bibitem[Qui69]{quillen}
{\scshape Quillen, D.} Rational homotopy theory. {\em Ann. of Math. (2)} {\bf 90} (1969) 205--295. \mrev{0258031} (41 \#2678), \zbl{0191.53702}.

\bibitem[Rez02]{rezkSA} 
{\scshape Rezk, C.} Every homotopy theory of simplicial algebras admits a proper model. {\em Topology Appl.} {\bf 119} (2002) no. 1, 65--94. \mrev{1881711} (2003g:55033), \zbl{0994.18008}.

\bibitem[Rez96]{rezk}
{\scshape Rezk, C.} Spaces of algebra structures and cohomology of operads. Ph.D. thesis, MIT, 1996.




\bibitem[SS00]{ss}
{\scshape Schwede, S.; Shipley, B.} Algebras and modules in monoidal model categories. {\em Proc. London Math. Soc.} {\bf 80} (2000) 491--511. \mrev{1734325} (2001c:18006), \zbl{1026.18004}.


\bibitem[SS03]{ss03}
{\scshape Schwede, S.; Shipley, B.} Equivalences of monoidal model categories. {\em Alg. Geom. Top.} {\bf 3} (2003) 287--334. \mrev{1997322} (2004i:55026), \zbl{1028.55013}.


\bibitem[Spi01]{spitzweck-thesis}
{\scshape Spitzweck, M.}  Operads, algebras and modules in general model categories. Preprint, 2001, \href{http://arxiv.org/abs/math/0101102}{arXiv:math/0101102} [math.AT].


\bibitem[Wei94]{weibel}
{\scshape Weibel, C.A.} An Introduction to Homological Algebra. Cambridge Studies in Advanced Mathematics 38. {\em Cambridge University Press, Cambridge, MA,} 1994. \mrev{1269324} (95f:18001), \zbl{0797.18001}.

\bibitem[WY15]{white-yau}
{\scshape White, D.; Yau, D.} Bousfield localization and algebras over colored operads. Preprint, 2015, \href{https://arxiv.org/abs/1503.06720}{arXiv:1503.06720} [math.AT]. 


\bibitem[Yau16]{yau-operad}
{\scshape Yau, D.} Colored Operads. Graduate Studies in Math., 170. {\em Amer. Math. Soc., Providence, RI}, 2016. \mrev{3444662}, \zbl{1348.18014}.


\bibitem[YJ15]{jy2}
{\scshape Yau, D.; Johnson, M.W.} A Foundation for PROPs, Algebras, and Modules. Math. Surveys and Monographs 203. {\em Amer. Math. Soc., Providence, RI}, 2015. \mrev{3329226}, \zbl{1328.18014}.


\end{thebibliography}
\end{document}